%% file: main.tex
\documentclass[11pt, twoside]{article}
\usepackage[utf8]{inputenc}
\usepackage[margin=2.5cm]{geometry} 
\usepackage{graphicx}
\usepackage{soul}
\usepackage{lipsum}
\usepackage{fancyhdr} 
\usepackage{array}
\usepackage{tikz}
\usepackage{amsmath,amssymb,amsfonts}
\usepackage{algorithm}
\usepackage[export]{adjustbox}
\usepackage[absolute,overlay]{textpos}

\usepackage{verbatim}
\usepackage{multicol}
\usepackage{mathtools}
\usepackage{mathrsfs}
\usepackage{subcaption}
\usepackage{bm}
\usepackage{amsthm}
\usepackage{gensymb}
\usepackage{enumerate}
\usepackage{colortbl}
\usepackage{algorithmic}
\usepackage[pdfpagemode={UseOutlines},bookmarks=true,bookmarksopen=true,
   bookmarksopenlevel=0,bookmarksnumbered=true,hypertexnames=false,
   colorlinks,linkcolor={blue},citecolor={blue},urlcolor={red},
   pdfstartview={FitV},unicode,breaklinks=true]{hyperref}
\definecolor{lightgray}{rgb}{0.85,0.85,0.85}

\newcommand{\cC}			{{\mathscr{C}}}
\newcommand{\cU}			{{\mathscr{U}}}

\newcommand{\cA}            {{\mathcal{A}}}
\newcommand{\cB}            {{\mathcal{B}}}

\newcommand{\cZ}            {{\mathscr{Z}}}
\newcommand{\grad}		{{\nabla}}

\newcommand{\nn}			{{\nonumber}}
\newcommand{\calR}            {{\mathcal{R}}}
\newcommand{\calQ}            {{\mathcal{Q}}}
\newcommand{\R}			{{\mathbb{R}}}
\newcommand{\U}         {{\mathbb{U}}}

\newcommand{\B}         {{\mathbb{B}}}
\newcommand{\N}         {{\mathbb{N}}}

\newcommand{\I}         {{\mathbb{I}}}
\newcommand{\spd}[1]    {\mathbb{S}^{#1}_{>0}}
\newcommand{\spsd}[1]    {\mathbb{S}^{#1}_{\geq 0}}
\newcommand{\sym}[1]         {\mathbb{S}^{#1}}

\newcommand{\norm}[1]{\left\lVert#1\right\rVert}

\DeclareMathOperator*{\argmax}{arg\,max}
\newcommand{\norml}[1]{\lVert#1\rVert_2}
\newcommand{\bnorml}[1]{\big\lVert#1\big\rVert_2}
\newcommand{\iprod}[2]{\left\langle#1, #2\right\rangle}
\newcommand{\half}[0]{\frac{1}{2}}
\newcommand{\integral}[4]{\int_{#1}^{#2}{#3}\,d#4}

\newtheorem{definition}{Definition}
\newtheorem{theorem}{Theorem}
\newtheorem{corollary}{Corollary}
\newtheorem{remark}{Remark}
\newtheorem{assumption}{Assumption}
\newtheorem{lemma}{Lemma}

\title{\textbf{A Hamilton-Jacobi-Bellman Approach to Ellipsoidal Approximations of Reachable Sets for Linear Time-Varying Systems}\thanks{This research was supported by the Australian Commonwealth Government through the Ingenium Scholarship, the Australian Research Council through a Linkage Project grant (Grant number: LP190100104) in collaboration with BAE systems, and the Air Force Office of Scientific Research (Grant number: FA2386‐22-1-4074).}}
\author{Vincent Liu$^\dag$ \and Chris Manzie$^\dag$ \and Peter M. Dower\thanks{The authors are with the Department of Electrical and Electronic Engineering, University of Melbourne, VIC, 3010 Australia (email: liuv2@student.unimelb.edu.au; manziec@unimelb.edu.au; pdower@unimelb.edu.au). }}
\date{}

\begin{document}
\maketitle

\begin{abstract}
Reachable sets for a dynamical system describe collections of system states that can be reached in finite time, subject to system dynamics. They can be used to guarantee goal satisfaction in controller design or to verify that unsafe regions will be avoided. However, conventional grid-based methods for computing these sets suffer from the curse of dimensionality, which typically prohibits their use for systems with more than a small number of states, even if they are linear. In this paper, we demonstrate that local viscosity supersolutions and subsolutions of a Hamilton-Jacobi-Bellman equation can be used to generate, respectively, under-approximating and over-approximating reachable sets for time-varying nonlinear systems. Based on this observation, we derive dynamics for a union and intersection of ellipsoidal sets that, respectively, under-approximate and over-approximate the reachable set for linear time-varying systems subject to an ellipsoidal input constraint and an ellipsoidal terminal (or initial) set. The dynamics for these ellipsoids can be selected to ensure that their boundaries coincide with the boundary of the exact reachable set along a collection of solutions of the system. The ellipsoids can be generated with polynomial computational complexity in the number of states, making our approximation scheme computationally tractable for continuous-time linear time-varying systems of relatively high dimension.

\vspace{1em}
\end{abstract}

\input{text/introduction}
\input{text/notation}
\input{text/background}
\input{text/approximate_reachability}
\input{text/LTV_reachability}
\input{text/numerical_examples}
\input{text/conclusion}
\input{text/appendices}

\bibliographystyle{IEEEtran}
\small{\bibliography{references}}
\end{document}

%% file: text/introduction.tex
\section{Introduction}
\label{sec: introduction}
The development of self-driving cars, the use of collaborative robots in industrial settings, and the deployment of drones in disaster responses are examples of society's ever-increasing reliance and integration of autonomous systems in day-to-day life. \emph{How can we be certain that these systems will complete their intended task and is it possible for them to behave unsafely?} It is such concerns that give impetus to research into methods of verifying the safety and reliability of these systems, and reachable sets lend themselves well to this task. When computed backwards in time, these sets describe all initial states for which there exists a constraint admissible control that leads the system to a set $\mathcal{X}$ at a terminal time $T$ (see Fig. \ref{fig: intro: reachable sets diagram}). These sets can be used to design controllers with rigorous guarantees of goal satisfaction, which have been previously deployed for maneuvering autonomous cars \cite{schurmann:21} and lane merging of unmanned aerial vehicles \cite{chen2017reachability}. When computed forwards in time, we obtain the set of states that can be reached using an admissible control starting from $\mathcal{X}$ at some initial time $t$. This information can be used to verify that an autonomous system will not enter an unsafe region due to an adversarial input, which is useful in developing collision avoidance control strategies \cite{llanes2022safety} or guaranteeing the existence of collision-free paths \cite{zhou2015reachable}. 
\begin{figure}[h]
    \centering
    \includegraphics[width=0.85\columnwidth]{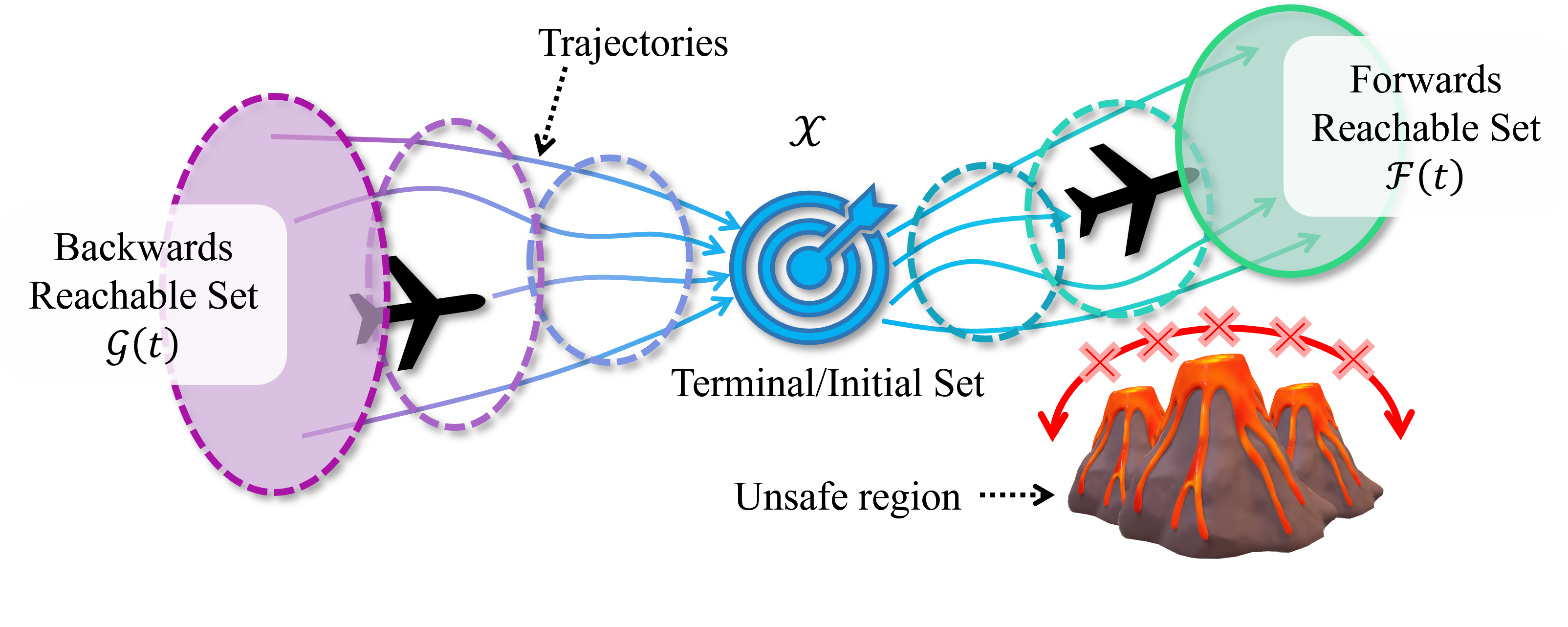}
    \vspace{-0.4cm}
    \caption{Diagram depicting the backwards reachable set $\mathcal{G}(t)$ and the forwards reachable set $\mathcal{F}(t)$, which enclose all admissible trajectories heading to or away from the set $\mathcal{X}$.}
    \label{fig: intro: reachable sets diagram}
\end{figure}

Hamilton-Jacobi (HJ) reachability frames the computation of reachable sets as an optimal control problem. It is a versatile formulation that allows for the direct treatment of nonlinear systems \cite{chen2018hamilton}, two-player game settings \cite{mitchell2005time}, and handling of state constraints \cite{altarovici2013general}. To perform reachability analysis, publicly available toolboxes, such as \cite{mitchelltoolbox}, provide a conventional and general-purpose method of numerically solving HJ equations over a grid in state-space. However, these approaches exhibit computational times that scale exponentially with the number of system states, rendering such schemes computationally intractable for many engineering applications. Consequently, a large body of research in this field focuses on efficient ways to compute or approximate reachable sets. 

When dealing with linear systems, a range of approximation schemes are available. In a discrete-time setting, where the terminal or initial set is a polytope and the input set is also a polytope, the reachable sets have an exact polytopic representation \cite[Chapter 10]{borrelli2017predictive}. However, the number of linear inequalities required to represent these polytopes grow with each time step, leading to issues with storage capacity and memory requirements. The use of zonotopes in \cite{yang:21} and \cite{girard2005reachability} can mitigate some of these issues due to their highly efficient set operations. Ellipsoidal approximation schemes have also garnered popularity in both discrete-time \cite{halder2020smallest} and continuous-time settings \cite{shishido:00} as ellipsoids can be represented with a fixed number of parameters. Generalisations of these geometries have also been studied, such as ellipsotopes \cite{kousik2022ellipsotopes}, which subsume both ellipsoids and zonotopes, and in characterisations based on support functions \cite{le2010reachability}. The linchpin in the aforementioned research is that approximations of the Minkowski sum can be used to develop under- and over-approximation (subset and superset) guarantees for reachability of linear systems. 

When it comes to nonlinear systems, linearisation-based techniques such as the Taylor series approximation \cite{althoff:13} or Koopman operator linearisation \cite{goswami2021bilinearization}, are common methods for leveraging existing algorithms for linear systems. Otherwise, efficient algorithms may be available, for instance, when the system can be decomposed into subsystems \cite{ChenDecomposition:18} or when a mixed monotone embedding exists \cite{coogan2020mixed}. Other approximation methods include implicit storage of the sets using classification-based control laws \cite{rubies2019classification}, neural network approximations \cite{julian2021reachability}, and boundary approximations \cite{alla2023tree}.

In this paper, we demonstrate that local viscosity supersolutions and subsolutions of a Hamilton-Jacobi-Bellman partial differential equation (HJB PDE) can be used to produce under- and over-approximations of reachable sets for time-varying nonlinear systems in continuous-time. This observation is then used to develop under- and over-approximation schemes for linear time-varying (LTV) systems with ellipsoidal input constraints and ellipsoidal terminal/initial sets. We focus on a single player setting, which is a special case of the problem considered in \cite{mitchell2005time}. Our approach employs the union and intersection of a family of ellipsoids to approximate reachable sets. These results extend our prior work in \cite{LMD23}, whereby classical supersolutions and subsolutions were used to develop ellipsoidal approximation schemes for linear time-invariant systems. By using viscosity solutions as opposed to classical solutions, approximating sets with non-smooth boundaries can be considered, for example, when the approximating set is constructed via the union or intersection of sets with smooth boundaries. A similar idea is utilised in \cite{xue2019inner}, albeit with smooth polynomial supersolutions developed for time-invariant polynomial systems.

A key novelty in our approach, relative to the vast majority of existing literature for linear systems, is that we do not rely on approximations of set operations (see e.g. \cite{kurzhanski2000ellipsoidal, kurzhanski2002reachability}). In particular, we need not deviate from the HJB framework used in nonlinear reachability analysis, thereby allowing for a unified treatment of both nonlinear and linear systems. The approximating sets developed through supersolutions and subsolutions of the HJB PDE possess the same desirable characteristics as existing schemes. For instance, the proposed under- and over-approximating collection of ellipsoids are also, respectively, under- and over-approximating for all intermediate times, which corresponds to the recurrence relation discussed in \cite{kurzhanski2000ellipsoidal}. The dynamics of each ellipsoid can be chosen to guarantee that their boundaries coincide with the exact reachable set along a solution of the system, which recovers the tightness property discussed in \cite{kurzhanski2002reachability}. The approach taken in this work yields a different characterisation for the ellipsoidal dynamics compared to the aforementioned references. In all cases, the integration of a matrix differential equation is required. However, this involves only polynomial (in the number of states) computational complexity. In addition to these properties, we demonstrate that the approximating ellipsoids inherit the recursive feasibility and recursive infeasibility properties of the exact reachable set. Specifically, if the system lies within the under-approximating set at a given time, then there exists a constraint admissible control that leads the system to the under-approximating set for any future time. Analogously, if the system is initialised outside the over-approximating set, then it is impossible for the system to be driven into the over-approximating set for any future time. Thus, these approximating sets preserve the intended purpose of reachable sets in control applications. 

The remainder of this paper is organised as follows. In Section \ref{sec: background theory}, we review the necessary preliminaries on HJB reachability analysis. In Section \ref{sec: approximate reachability}, we generate approximate reachability results for time-varying nonlinear systems. We then specialise to LTV systems in Section \ref{sec: LTV reachability}, leading to the development of ellipsoidal approximation schemes. Illustrative numerical examples are then considered in Section \ref{sec: numerical examples}.

%% file: text/notation.tex
\section*{Notation}
\begin{itemize}
    \item  Let $\N\doteq \{1,2,\ldots\}$ denote the set of natural numbers and let $\R_{>0}$ denote the set of positive real numbers. 
    \item We use $\langle \cdot, \cdot\rangle$ to denote the Euclidean inner product and $\norm{\cdot}_2\doteq \iprod{\cdot}{\cdot}^\half$ to denote the corresponding norm. 
    \item We use $\norm{\cdot}$ to denote any norm of a vector.
    \item  We use $\cA^c, \text{int}\left(\cA\right), \text{cl}\left(\cA\right)$, and $\partial \cA$ to denote the complement, interior, closure, and boundary of a set $\cA$, respectively.
    \item We say that set $\cA$ under-approximates (resp. over-approximates) set $\cB$ if $\cA$ is a non-strict subset (resp. superset) of $\cB$, i.e., $\cA\subseteq \cB$ (resp. $\cA \supseteq \cB$). 
    \item The space of $k$-times continuously differentiable functions from $\Omega$ to $\cA$ is denoted by $\cC^k(\Omega;\,\cA)$ with $\cC(\Omega;\,\cA) \doteq \cC^0(\Omega;\,\cA)$ denoting the space of continuous functions from $\Omega$ to $\cA$. Unless stated otherwise, if $\cA$ is a matrix space, continuity is defined with respect to the induced Euclidean norm.
    \item The pre-image of a set $\cA_0\subset \cA$ under $v:\Omega \rightarrow \cA$ is denoted by $v^{-1}\left(\cA_0 \right) \doteq \{ \omega \in \Omega\, |\, v(\omega) \in \cA_0 \}$.
    \item An open ball of radius $R > 0$ centred at $x_0 \in \R^n$ is denoted by $\B_R(x_0)\doteq \{x\in\R^n\,|\,\norm{x - x_0} < R\}$. 
    \item  We use $A \succ  0$ (resp. $A\succeq 0$) to say that a matrix $A\in\R^{n\times n}$ is positive definite (resp. semi-definite).
    \item The set of symmetric, symmetric positive definite, and symmetric positive semi-definite matrices in $\R^{n\times n}$ is denoted by $\sym{n}$, $\spd{n}$, and $\spsd{n}$, respectively.
    \item The set of orthogonal matrices in $\R^{n\times n}$ is denoted by $\mathbb{O}^n$.
    \item  We use $\mathbb{I}$ to denote the identity matrix whose size can be inferred from context.
    \item We denote by $\underline{\Lambda}(A)$ the smallest eigenvalue of $A\in\sym{n}$.
    \item Finally, the square root of $A \in \spsd{n}$ is a square matrix denoted by $A^{\half}\in \spsd{n}$ and satisfies $A = A^{\half}A^{\half}$.
\end{itemize}

%% file: text/background.tex
\section{Background Theory}
\label{sec: background theory}
\subsection{Hamilton-Jacobi-Bellman Reachability Analysis}
\label{subsec: HJ reachability background}
Fix $T \geq t \geq 0 $, $x\in\R^n$, and consider a finite dimensional time-varying nonlinear system evolving in continuous time, with trajectories satisfying the Cauchy problem
\begin{equation}
    \begin{split}
    \dot{x}(s) &= f(s, x(s), u(s)), \quad \text{a.e. } s \in (t,T),\\
    x(t) &= x,
    \end{split}
    \label{eq: background: nonlinear system}
\end{equation}
where $x(s)\in\mathbb{R}^n$ is the state and $u(s)\in\U$ is the input, both at time $s$, with $\U\subset \mathbb{R}^m$ assumed to be compact and non-empty. The control input is selected such that $u\in\cU[t,T]$, where 
\begin{equation}
    \cU[t,T] \doteq \left\{u:[t,T]\rightarrow \mathbb{U}\;| \;u \text{ measurable} \right\}.
    \label{eq: background: admissible controls}
\end{equation}
Additionally, we assume the following for $f$ throughout. 
\begin{assumption}
The function $f:[0,T]\times\mathbb{R}^n\times\mathbb{U}\rightarrow\mathbb{R}^n$ satisfies
\begin{enumerate}[i)]
\item \label{item: background: flow field conditions item 1} $f\in\cC\left([0,T]\times\mathbb{R}^n\times\mathbb{U};\, \mathbb{R}^n\right)$; and
    \item \label{item: background: flow field conditions item 2} $f$ is Lipschitz continuous in $x$, uniformly in $s$ and $u$, that is, there exists a Lipschitz constant $L_{f} \geq 0$ such that $\norm{f(s,x,u)- f(s,y,u)} \leq L_{f} \norm{x-y}$ for all $(s,x,y,u)\in[0,T]\times\R^n\times\R^n\times\U$.
\end{enumerate}
\label{assumption: background: flow field conditions}
\end{assumption}

Under Assumption \ref{assumption: background: flow field conditions}, the dynamics in \eqref{eq: background: nonlinear system} admit a unique and continuous solution for any fixed control $u\in\cU[t,T]$ (see \cite[Theorem 3.2, p.93]{khalil2002nonlinear}). We denote these solutions at time $s \in [t,T]$ by $\varphi(s ;t, x, u)$. In order to consider closed terminal/initial sets, an additional assumption for compactness of solutions of \eqref{eq: background: nonlinear system} is required, see for example \cite[Theorem 7.1.6 and Remarks 7.1.7-7.1.8, pp.190-191]{CS:04}. 
\begin{assumption} For any $x\in\R^n$ and $s\in[0,T]$, the set $f(s, x, \U) \doteq \{f(s,x,u)\,|\, u \in \U\}$ is convex.
    \label{assumption: background: convex flow field}
\end{assumption}
\begin{remark}
    If \eqref{eq: background: nonlinear system} is control affine, that is, $f:[0,T]\times\R^n\times\U\rightarrow\R$ can be decomposed as $f(s,x,u) = f_1(s,x) + f_2(s,x)u$ and $\U$ is convex then Assumption \ref{assumption: background: convex flow field} is satisfied. 
    \label{remark: background: convexity of control affine systems}
\end{remark}

The backwards reachable set for \eqref{eq: background: nonlinear system}, which we denote by $\mathcal{G}(t)$, is defined as the set of states at time $t$ for which there exists a control $u\in\cU[t,T]$ that leads the system to a set $\mathcal{X}$ at the terminal time $T$. We assume that $\mathcal{X}$ can be described by the zero sublevel set of a real-valued function, in particular,
\begin{equation}
    \mathcal{X} \doteq \left\{x\in\mathbb{R}^n \, |\, g(x) \leq 0 \right\},
    \label{eq: background: closed terminal set}
\end{equation}
where $g\in\cC(\R^n;\,\R)$. We define $\mathcal{G}(t)$ more precisely below.
\begin{definition} Consider the terminal set $\mathcal{X}$ defined in \eqref{eq: background: closed terminal set}. The backwards reachable set $\mathcal{G}(t) = \mathcal{G}(t;\mathcal{X})$ of \eqref{eq: background: nonlinear system} at time $t\in[0,T]$ is defined by 
\begin{equation}
    \mathcal{G}(t)  \doteq \left\{x\in\mathbb{R}^n \, | \, \exists\, u\in\cU[t,T] \text{ s.t. } \varphi(T; t, x, u) \in \mathcal{X} \right\}.
    \label{eq: background: BRS}
\end{equation}
\vspace{-1em}
\label{def: background: BRS}
\end{definition}

HJB reachability characterises $\mathcal{G}(t)$ as the zero sublevel set of a value function $v: [0,T]\times \mathbb{R}^n\rightarrow \R$ corresponding to the Mayer optimal control problem given by
\begin{equation}
    v(t,x) \doteq \inf_{u\in\cU[t,T]}g\left(\varphi(T; t, x, u)\right),
    \label{eq: background: value function}
\end{equation}
which in turn is the viscosity solution to a HJB PDE \cite[Theorem 7.2.4, p.194]{CS:04}. Before we state this result, let us attach to \eqref{eq: background: nonlinear system} a Hamiltonian $H:[0,T]\times\mathbb{R}^n\times \mathbb{R}^n \rightarrow \mathbb{R}$ given by
\begin{equation}
    H(t,x,p)\doteq \max_{u\in\U}\iprod{-p}{f(t,x, u)}.
    \label{eq: background: HJB Hamiltonian}
\end{equation}
We recall the disturbance-free case of \cite[Lemma 8]{mitchell2005time} below. 
\begin{theorem} Let $f:[0,T]\times\R^n\times\U\rightarrow\R$ satisfy Assumptions \ref{assumption: background: flow field conditions} and \ref{assumption: background: convex flow field}. Consider the HJB PDE given by 
\begin{equation}
\begin{split}
    -v_t + H(t, x, \nabla v) &= 0, \quad \forall (t,x)\in (0,T)\times\mathbb{R}^n, \\
    v(T,x) &= g(x), \quad \forall x\in\mathbb{R}^n,
    \end{split}
    \label{eq: background: HJB PDE with terminal condition BRS}
\end{equation}
where $g\in\cC\left(\R^n;\,\R\right)$ defines the terminal set $\mathcal{X}$ in \eqref{eq: background: closed terminal set}, and the Hamiltonian $H$ is defined as in \eqref{eq: background: HJB Hamiltonian}. Then, \eqref{eq: background: HJB PDE with terminal condition BRS} admits the unique viscosity solution $v\in\cC([0,T] \times\mathbb{R}^n\,;\mathbb{R})$ given by \eqref{eq: background: value function} and the backwards reachable set for \eqref{eq: background: nonlinear system} at time $t\in[0,T]$ is
\begin{equation*}
    \mathcal{G}(t) = \left\{x\in\mathbb{R}^n \; | \; v(t, x) \leq 0\right\}.
\end{equation*}
\vspace{-1em}
\label{theorem: background: HJB equation for BRS}
\end{theorem}
Note that uniqueness of solutions of \eqref{eq: background: HJB PDE with terminal condition BRS} follows from \cite[Theorem 3.17, p.159]{bardi1997optimal} with the extension to the non-autonomous setting following from \cite[Remark 3.10, pp.154-155]{bardi1997optimal} and \cite[Exercises 3.6-3.7, pp.182-183]{bardi1997optimal}.

An analogous set to the backwards reachable set $\mathcal{G}(t)$ is the forwards reachable set $\mathcal{F}(t)$, which characterises all states that can be reached by time $T\geq t$ from the set $\mathcal{X}$ starting at time $t$ under the influence of a constraint admissible control. The forwards reachable set is more precisely defined below.
\begin{definition} Consider the terminal set $\mathcal{X}$ defined in \eqref{eq: background: closed terminal set}. The forwards reachable set $\mathcal{F}(t) = \mathcal{F}(t;\mathcal{X})$ of \eqref{eq: background: nonlinear system} at time $t\in[0,T]$ is defined by
\begin{equation}
    \mathcal{F}(t)  \doteq \{\varphi(T; t, x, u)\in\mathbb{R}^n \, | \, u\in\cU[t,T] \text{ and } x \in \mathcal{X} \}.
    \label{eq: background: FRS}
\end{equation}
\vspace{-1em}
\label{def: background: FRS}
\end{definition}

It turns out that the forwards reachable set is exactly the backwards reachable set of the time-reversed system given by
\begin{equation}
\begin{split}
    \dot{x}(s) &= -f(T+t-s, x(s), u(s)), \quad\text{a.e. }s \in (t,T), \\
    x(t) &= x,
\end{split}
    \label{eq: background: reversed nonlinear system}
\end{equation}
where $f$ is given in \eqref{eq: background: nonlinear system}. The following theorem states this relationship. We omit the proof for brevity, but the autonomous case appears in Proposition 6 of \cite{mitchellforwardandbackward:07} and similar arguments to Theorem 4 of \cite{alla2023tree} can be used to demonstrate this result. 
\begin{lemma}
Let $\mathcal{G}_-(t)$ denote the backwards reachable set defined in \eqref{eq: background: BRS} for the time-reversed system in \eqref{eq: background: reversed nonlinear system} and let $\mathcal{F}(t)$ denote the forwards reachable set as defined in \eqref{eq: background: FRS}. Then, $\mathcal{F}(t) = \mathcal{G}_-(t)$ for all $t\in[0,T]$.
\label{lemma: background: FRS from BRS}
\end{lemma}
The remainder of this paper will focus on the backwards reachable set, noting that analogous results for forwards reachability can be produced via Lemma \ref{lemma: background: FRS from BRS}.

\subsection{Viscosity Solutions of HJB Equations}
\label{subsec: viscosity solutions background}
To handle approximating reachable sets that are characterised by potentially non-smooth manifolds, we will need to consider supersolutions and subsolutions of \eqref{eq: background: HJB PDE with terminal condition BRS} in the viscosity sense. Accordingly, we recall the necessary background and definitions associated with viscosity solutions, which we adapt from \cite{crandall1983viscosity}. Consider the HJB equation given by
\begin{equation}
    -v_t + H(t,x,\nabla v) = 0, \quad\forall (t,x) \in \Omega,
    \label{eq: background: HJ}
\end{equation}
where $T >  0$ is a priori fixed, $\Omega\subseteq (0,T)\times \mathbb{R}^n$ is an open set, and $H:\Omega \times \mathbb{R}^n \rightarrow \mathbb{R}$ is a continuous function (which corresponds to \eqref{eq: background: HJB Hamiltonian} in our work). 
\begin{definition}
A function $v\in \cC\left([0,T]\times\R^n;\,\mathbb{R}\right)$ is a continuous viscosity supersolution (resp. subsolution) of \eqref{eq: background: HJ} if, for every $\xi\in \cC^1\left(\Omega;\,\mathbb{R}\right)$ that results in $v-\xi$ attaining a local minimum (resp. local maximum) at $(t_0,x_0)\in \Omega$, 
\begin{equation}
    -\xi_t(t_0,x_0) + H(t_0,x_0,\nabla \xi(t_0,x_0)) \geq 0 \quad (\text{resp. } \leq 0).
    \label{eq: background: visc supersolution and subsolution}
\end{equation}
Additionally, we say that $v$ is a viscosity solution of \eqref{eq: background: HJ} if it is both a viscosity supersolution and subsolution of \eqref{eq: background: HJ}.
\label{def: background: visc supersolution and subsolution}
\end{definition}

The definitions given in \eqref{eq: background: visc supersolution and subsolution} utilise a variational description of Fr{\'e}chet subdifferentials and superdifferentials (whose elements are called Fr{\'e}chet subgradients and supergradients). Consequently, it is possible to define viscosity supersolutions and subsolutions of \eqref{eq: background: HJ} in terms of these subdifferentials and superdifferentials, which are defined below according to \cite{kruger2003frechet}.
\begin{definition}
A vector $p\in\R^{n+1}$ is a subgradient (resp. supergradient) of $v\in\cC\left([0,T]\times\R^n;\,\mathbb{R}\right)$ at $(t_0, x_0)\in\Omega$ if\small
\begin{align*}
    &\liminf_{(t,x)\rightarrow(t_0,x_0)}\frac{v(t,x)-v(t_0,x_0)-\iprod{p}{(t,x)-(t_0,x_0)}}{\norm{(t,x)-(t_0,x_0)}}  \geq 0, \\
    &\limsup_{(t,x)\rightarrow(t_0,x_0)}\frac{v(t,x)-v(t_0,x_0)-\iprod{p}{(t,x)-(t_0,x_0)}}{\norm{(t,x)-(t_0,x_0)}}\leq 0. \quad (\text{resp.})
\end{align*}\normalsize
The set of all such $p\in\R^{n+1}$ satisfying the above is called the subdifferential (resp. superdifferential) of $v$ at $(t_0, x_0)$, which we denote by $D^-v(t_0, x_0)$ (resp. $D^+v(t_0,x_0)$).
\label{def: background: reg subdifferentials and superdifferentials}
\end{definition}
The subdifferential (resp. superdifferential) of $v$ at a point characterise the set of gradients that any continuously differentiable function $\xi$ can take at that point if $\xi$ were to touch $v$ from below (resp. above). This result comes from \cite[Proposition 8.5, p.302]{rockafellar2009variational}, which we state in the following.
\begin{lemma} A vector $p\in\R^{n+1}$ is an element of the subdifferential (resp. superdifferential) of $v\in\cC\left([0,T]\times\R^n;\,\mathbb{R}\right)$ at $(t_0, x_0)\in\Omega$ if and only if there exists a function $\xi\in\cC^1\left(\Omega;\,\R\right)$ such that $\left(\xi_t(t_0,x_0), \grad\xi(t_0, x_0)\right) = p$ and $v - \xi$ attains a local minimum (resp. local maximum) at $(t_0, x_0)$.
\label{lemma: background: variational description of sub and superdifferentials}
\end{lemma}

Finally, we recall a classic global comparison result for viscosity supersolutions and subsolutions of \eqref{eq: background: HJB PDE with terminal condition BRS} from \cite[Theorem 3.15, p.158]{bardi1997optimal}. We provide the time-varying version below for the backwards HJB PDE (see \cite[Remark 3.10, pp.154-155]{bardi1997optimal} and \cite[Exercises 3.6-3.7, pp.182-183]{bardi1997optimal}).
\begin{theorem} Let $f$ in \eqref{eq: background: nonlinear system} satisfy Assumption \ref{assumption: background: flow field conditions} and let $v_1,v_2\in\cC([0,T]\times\R^n;\,\R)$ be, respectively, a viscosity supersolution and subsolution of \eqref{eq: background: HJ} over the domain $\Omega = (0,T)\times\R^n $ with the Hamiltonian $H$ given by \eqref{eq: background: HJB Hamiltonian}. If $v_1(T,x) \geq v_2(T,x)$ for all $x\in \R^n$, then, $v_1(t,x) \geq v_2(t,x)$ for all $(t,x)\in \rm{cl}\left(\Omega\right)$.
\label{theorem: background: global comparison result}
\end{theorem}

%% file: text/approximate_reachability.tex
\section{Approximating Reachable Sets}
\label{sec: approximate reachability}
With the preliminaries introduced, we will now demonstrate that viscosity supersolutions and subsolutions of \eqref{eq: background: HJB PDE with terminal condition BRS} defined over a local domain can be used to produce under- and over-approximating backwards reachable sets for \eqref{eq: background: nonlinear system}. Before doing so, let us motivate and highlight the main idea behind this approach. Consider the functions $\overline{v}\in \cC([0,T]\times\R^n;\,\R)$ and $\underline{v}\in\cC([0,T]\times\R^n;\,\R)$, which are, respectively, a viscosity supersolution and subsolution of  
\begin{equation}
    -v_t + H(t,x,\nabla v) = 0, \quad \forall(t,x)\in\Omega, 
    \label{eq: approx reach: HJB PDE}
\end{equation}
where $\Omega \doteq (0,T)\times\R^n$ and $H$ is given by \eqref{eq: background: HJB Hamiltonian}. From Theorem \ref{theorem: background: global comparison result}, it follows that if $\overline{v}(T,x)\geq g(x) \geq \underline{v}(T,x)$ for all $x\in\R^n$, then, $\overline{v}(t,x) \geq v(t,x) \geq \underline{v}(t,x)$ for all $(t,x)\in \text{cl}\left(\Omega\right)$ where the value function $v:[0,T]\times\R^n\rightarrow\R$ as defined in \eqref{eq: background: value function} is the continuous viscosity solution of \eqref{eq: approx reach: HJB PDE}. Consequently, the zero sublevel set of $x\mapsto\overline{v}(t,x)$ and $x\mapsto \underline{v}(t,x)$ can then be used to produce under- and over-approximations of the backwards reachable set. In particular,
\begin{align*}
    \{x\in\R^n\,|\,\overline{v}(t,x) \leq 0\} \subseteq \overbrace{\{x\in\R^n\,|\,v(t,x) \leq 0\}}^{\mathcal{G}(t)} \subseteq \{x\in\R^n\,|\,\underline{v}(t,x) \leq 0\}.
\end{align*}
However, the issue with using \eqref{eq: approx reach: HJB PDE} is that it imposes the supersolution and subsolution condition for all points in $\Omega$, which can be overly restrictive since $\mathcal{G}(t)$ is characterised by points \emph{only} in the zero sublevel set of $v$. The remainder of this section explores the case when $\Omega$ is chosen as either of the pre-images $\overline{v}^{-1}((-\infty ,0))$ or $\underline{v}^{-1}((0 ,\infty))$. This will mean that $\overline{v}$ and $\underline{v}$ form upper and lower bounds for $v$ only over their zero sublevel set and superlevel set, respectively. 

\subsection{Under-approximating Backwards Reachable Sets}
Let us now demonstrate that viscosity supersolutions $\overline{v}$ of \eqref{eq: approx reach: HJB PDE} defined over its zero sublevel set can be used to produce under-approximating backwards reachable sets. Although it may be possible to proceed with local comparison results such as \cite[Theorem 9.1, p.90]{FS:06}, these results require a priori knowledge that $\overline{v}(t,x) \geq v(t,x)$ along the boundary of the local domain. Instead, we will demonstrate that a truncation of $\overline{v}$ from above by 0, which is defined by
\begin{equation}
    \overline{v}^{\times}(t,x) \doteq \min\{\overline{v}(t,x), \,0\}, \quad \forall (t,x)\in[0,T]\times\R^n,
    \label{eq: approx reach: global extension supersolution truncated value function}
\end{equation}
is a viscosity supersolution of \eqref{eq: approx reach: HJB PDE} over $(t,x)\in(0,T)\times\R^n$. Theorem \ref{theorem: background: global comparison result} can then be used to conclude that $\overline{v}^\times(t,x) \geq v(t,x)$, which also holds for $\overline{v}$ across all points in the \emph{strict} zero sublevel set of $x\mapsto\overline{v}^\times(t,x)$. We will require that points in the zero level set of $x\mapsto\overline{v}(t,x)$ lie arbitrarily close to points in its strict zero sublevel set. For convenience, we define a class of functions for which this property holds. 
\begin{definition} A function $z:\R^n\rightarrow \R$ is said to belong to class $\cZ$ if for every $\Bar{x}\in z^{-1}(\{0\})$ and for every $r>0$, there exists an $x_{<0}\in\B_r(\Bar{x})$ with $z(x_{<0})<0$.
\label{def: approx reach: zero regular continuous function}
\end{definition}

Class $\cZ$ functions have the property that the closure of its strict zero sublevel set is given by its non-strict zero sublevel set and the interior of the latter set is its strict zero sublevel set. Note also that if the terminal data $g$ is of class $\cZ$ then so is the value function in \eqref{eq: background: value function}, which is a result needed in subsequent sections. The following two lemmas state these results precisely with their proofs postponed to the Appendix. 

\begin{lemma}
If $z\in\cC(\R^n;\,\R)$ is of class $\cZ$, then\footnote{In general the zero level set of $z$ may include more points than required to close its strict zero sublevel set $z^{-1}((-\infty,0))$. This occurs, for instance, when the zero level set of $z$ exhibits flat regions or when the graph of $z$ just touches zero, which creates disjoint singletons in its zero level set.},
\begin{align}
\mathrm{cl}\left(z^{-1}((-\infty, 0))\right) &= z^{-1}((-\infty, 0]),\label{eq: approx reach: closure of strict sublevel set}\\
    \mathrm{int}\left(z^{-1}((-\infty, 0])\right) &= z^{-1}((-\infty, 0)).\label{eq: approx reach: interior of non-strict sublevel set}
\end{align}
\vspace{-1em}
\label{lemma: approx reach: closure and interior of sublevel sets}
\end{lemma}
\begin{lemma} Consider the value function $v\in\cC([0,T]\times\R^n;\,\R)$ defined in \eqref{eq: background: value function} with $g\in\cC(\R^n;\,\R)$. Let $f$ in \eqref{eq: background: nonlinear system} satisfy Assumptions \ref{assumption: background: flow field conditions} and \ref{assumption: background: convex flow field} and let $g$ be of class $\cZ$. Then, the map $x\mapsto v(t,x)$ is also of class $\cZ$ for all $t\in[0,T]$.
\label{lemma: approx reach: value function satisfies attached sublevel set}
\end{lemma}

With this, we now state our results on producing under-approximations of the backwards reachable set of \eqref{eq: background: nonlinear system}.
\begin{lemma} Let $\overline{v}\in\cC([0,T]\times\R^n;\,\R)$ be a viscosity supersolution of \eqref{eq: approx reach: HJB PDE} with $\Omega = \overline{v}^{-1}((-\infty,0))$. Then, $\overline{v}^\times\in \cC([0,T]\times\R^n ;\,\R)$ as defined in \eqref{eq: approx reach: global extension supersolution truncated value function} is a viscosity supersolution of \eqref{eq: approx reach: HJB PDE} with $\Omega = (0,T)\times\R^n$ and the terminal condition $\overline{v}^{\times}(T,x) = \min\{\overline{v}(T,x),\, 0\}$ for all $x \in \R^n$.
    \label{lemma: approx reach: global extension for local supersolutions}
\end{lemma}
\begin{proof}
    Let $\xi\in\cC^1((0,T)\times\R^n;\,\R)$ be defined such that $\overline{v}^{\times} - \xi$ attains a local minimum at $(t_0, x_0) \in (0,T)\times\R^n$. Note that by Lemma \ref{lemma: background: variational description of sub and superdifferentials}, the set $D^-\overline{v}^{\times}(t_0,x_0)$ must be non-empty. Since $\overline{v}^{\times}(t,x) \leq 0$ for all $(t,x)\in(0,T)\times\R^n$, the following cases can be considered: case \emph{i)} $(t_0, x_0) \in \mathscr{F}^{\overline{v}}_{<0}\doteq (\overline{v}^{\times})^{-1}((-\infty,0)) $ or case \emph{ii)} $(t_0, x_0) \in\mathscr{F}^{\overline{v}}_{0}\doteq (\overline{v}^{\times})^{-1}(\{0\})$.

    \underline{Case \emph{i)}}. Let $(t_0, x_0)\in\mathscr{F}^{\overline{v}}_{<0}$. Since $\overline{v}^{\times}(t,x) = \overline{v}(t,x)$ for all $(t,x)\in\overline{v}^{-1}((-\infty,0))$ and $\mathscr{F}^{\overline{v}}_{<0}$ is open, for some sufficiently small $r>0$, $\overline{v}^{\times}(t,x) = \overline{v}(t,x)$ for all $(t,x)\in\B_r\left((t_0,x_0)\right)$. Thus, $\overline{v}-\xi$ also attains a local minimum at $(t_0,x_0)$. As $\overline{v}$ is a viscosity supersolution of \eqref{eq: approx reach: HJB PDE}, by Definition \ref{def: background: visc supersolution and subsolution},
    \begin{equation}
        -\xi_t(t_0,x_0) + H(t_0,x_0,\nabla\xi(t_0,x_0)) \geq 0.
        \label{eq: approx reach: global extension supersolution proof 1}
    \end{equation}
    
   \underline{Case \emph{ii)}}. Let $(t_0,x_0)\in\mathscr{F}^{\overline{v}}_{0}$. From Lemma \ref{lemma: background: variational description of sub and superdifferentials}, $\left(\xi_t(t_0,x_0), \nabla\xi(t_0,x_0)\right) \in D^-\overline{v}^{\times}(t_0,x_0)$. Since the zero function must touch $\overline{v}^\times$ from above at $(t_0, x_0)$, $D^+\overline{v}^{\times}(t_0,x_0)$ is non-empty and contains the zero vector, which then implies that $D^-\overline{v}^{\times}(t_0,x_0)$ and $D^+\overline{v}^{\times}(t_0,x_0)$ coincide. To see this, note that for all $(t,x)\in(0,T)\times\R^n$, $\overline{v}^{\times}(t,x) - \overline{v}^{\times}(t_0,x_0) \leq 0$ as $\overline{v}^{\times}(t,x) \leq 0$ by \eqref{eq: approx reach: global extension supersolution truncated value function} and $\overline{v}^{\times}(t_0,x_0)=0$. Thus,
    \begin{align*}
        &\limsup_{(t,x)\rightarrow(t_0,x_0)}{ \frac{\overline{v}^{\times}(t,x) - \overline{v}^{\times}(t_0,x_0) - \iprod{(0,0)}{(t-t_0,x-x_0)}}{\norm{(t,x)-(t_0,x_0)}}} \leq 0.
    \end{align*} 
    Then, by Definition \ref{def: background: reg subdifferentials and superdifferentials}, $0\in D^+\overline{v}^{\times}(t_0,x_0)$. As $D^-\overline{v}^{\times}(t_0,x_0)$ is also non-empty, it follows that $D^-\overline{v}^{\times}(t_0,x_0)=D^+\overline{v}^{\times}(t_0,x_0) = \{0\}$ (see \cite[Proposition 3.1.5, p.51]{CS:04}). By Lemma \ref{lemma: background: variational description of sub and superdifferentials}, 
    \begin{equation*}
        \left(\xi_t(t_0,x_0), \nabla\xi(t_0,x_0)\right) = 0,
    \end{equation*}
which yields \eqref{eq: approx reach: global extension supersolution proof 1} with equality. It then follows from \eqref{eq: approx reach: global extension supersolution proof 1} and Definition \ref{def: background: visc supersolution and subsolution} that $\overline{v}^{\times}$ is a viscosity supersolution of \eqref{eq: approx reach: HJB PDE}. The terminal condition follows from the definition of $\overline{v}^{\times}$. 
\end{proof}
\begin{theorem} Consider the terminal set $\mathcal{X}$ defined in \eqref{eq: background: closed terminal set}. Let $f$ in \eqref{eq: background: nonlinear system} satisfy Assumptions \ref{assumption: background: flow field conditions} and \ref{assumption: background: convex flow field} and let $\overline{v}\in\cC([0,T]\times\R^n;\,\R)$ be a viscosity supersolution of \eqref{eq: approx reach: HJB PDE} with $\Omega = \overline{v}^{-1}((-\infty, 0))$ and the terminal condition $\overline{v}(T,x) \doteq \overline{g}(x)$ for all $ x\in \R^n$. If $x\mapsto \overline{v}(t,x)$ is of class $\cZ$ for all $t\in[0,T]$, and $\overline{g}(x) > 0$ for all $ x\notin \mathcal{X}$, then, the set
\begin{equation}
    \underline{\mathcal{G}}(t) \doteq \left\{x\in\mathbb{R}^n \; | \; \overline{v}(t, x) \leq 0\right\}
    \label{eq: approx reach: under approx BRS via HJB zero sublevel set}%
\end{equation}
is an under-approximation of $\mathcal{G}(t)$ for \eqref{eq: background: nonlinear system} for all $t\in[0,T]$.
\label{theorem: approx reach: HJB equation for under-approximate BRS}
\end{theorem}
\begin{proof}
 By Lemma \ref{lemma: approx reach: global extension for local supersolutions}, $\overline{v}^{\times}$ as defined in \eqref{eq: approx reach: global extension supersolution truncated value function}, is a viscosity supersolution of
\begin{equation}
    \begin{split}
        -v_t + H(t,x,\nabla v) &= 0, \quad \forall (t,x)\in (0,T)\times\R^n, \\
        v(T,x) = \min\{\overline{g}(x),\,0\} &\doteq \overline{g}^{\times}(x), \quad \forall x\in\R^n.
    \end{split}
    \label{eq: approx reach: HJB PDE under approx BRS proof 1}
\end{equation}
Consider the value function $\Tilde{v}\in\cC([0,T]\times\R^n;\,\R)$ defined by $\Tilde{v}(t,x) \doteq \inf_{u\in\cU[t,T]}\overline{g}^{\times}(\varphi(T;t,x,u))$ for all $(t,x)\in[0,T]\times\R^n$. By Theorem \ref{theorem: background: HJB equation for BRS}, $\Tilde{v}$ is the unique viscosity solution of \eqref{eq: approx reach: HJB PDE under approx BRS proof 1}, so $\Tilde{v}$ is also a viscosity subsolution of \eqref{eq: approx reach: HJB PDE under approx BRS proof 1} (see Definition \ref{def: background: visc supersolution and subsolution}). As $\Tilde{v}(T,x)= \overline{g}^{\times}(x)$ for all $x\in\R^n$, by Theorem \ref{theorem: background: global comparison result},
\begin{equation}
   \Tilde{v}(t,x) \leq  \overline{v}^{\times}(t,x), \quad \forall(t,x)\in[0,T]\times\R^n.
   \label{eq: approx reach: HJB PDE under approx BRS proof 2}
\end{equation}

Now, fix $\Bar{t}\in[0,T]$ and $\Bar{x} \in \{x\in\R^n\,|\,\overline{v}(\Bar{t},x) < 0\}$, which means $\overline{v}^{\times}(\Bar{t},\Bar{x}) = \min\{\overline{v}(\Bar{t},\Bar{x}),\,0\} < 0$. Then, from \eqref{eq: approx reach: HJB PDE under approx BRS proof 2}, $\Tilde{v}(\Bar{t},\Bar{x}) \leq \overline{v}^{\times}(\Bar{t},\Bar{x}) \doteq -\delta_{\Bar{x}} < 0$, which yields
\begin{align*}
    &\Tilde{v}(\Bar{t},\Bar{x})  = \inf_{u\in\cU[\Bar{t},T]}\min\{\overline{g}(\varphi(T;\Bar{t},\Bar{x},u)),\,0\} \leq -\delta_{\Bar{x}} \\
    &\quad \implies \min\left\{\inf_{u\in\cU[\Bar{t},T]}\overline{g}(\varphi(T;\Bar{t},\Bar{x},u)),\, 0\right\} \leq -\delta_{\Bar{x}} \\
    &\quad\implies \inf_{u\in\cU[\Bar{t},T]}\overline{g}(\varphi(T;\Bar{t},\Bar{x},u)) \leq -\delta_{\Bar{x}},
\end{align*}
where $-\delta_{\Bar{x}} < 0$ is used to obtain the final implication. As the infimum is the greatest lower bound, for any $\Tilde{\epsilon} > 0$, there must exist a $\Tilde{u}\in\cU[\Bar{t},T]$ such that
\begin{equation}
     \overline{g}(\varphi(T;\Bar{t},\Bar{x},\Tilde{u})) \leq  \inf_{u\in\cU[\Bar{t},T]}\overline{g}(\varphi(T;\Bar{t},\Bar{x},u)) + \Tilde{\epsilon}.
     \label{eq: approx reach: HJB PDE under approx BRS proof 3}
\end{equation}
Selecting $\Tilde{\epsilon} = \frac{\delta_{\Bar{x}}}{2}$ in \eqref{eq: approx reach: HJB PDE under approx BRS proof 3} yields $\overline{g}(\varphi(T;\Bar{t},\Bar{x},\Tilde{u}))< 0$.
The statement $\overline{g}(x) > 0$ for all $x\notin \mathcal{X}$ is equivalent to $\overline{g}(x) \leq 0 $ implies $ x \in \mathcal{X}$. Since $\overline{g}(\varphi(T;\Bar{t},\Bar{x},\Tilde{u}))< 0 $, then $\varphi(T;\Bar{t},\Bar{x},\Tilde{u})\in\mathcal{X}$, and $\Bar{x}\in\mathcal{G}(\Bar{t})$. As $\Bar{x}\in\{x\in\R^n\,|\,\overline{v}(\Bar{t},x) < 0\}$ is arbitrary, it follows that
$\{x\in\R^n\,|\,\overline{v}(\Bar{t},x) < 0\} \subseteq \mathcal{G}(\Bar{t})$.

Note that $\mathcal{G}(\Bar{t})$ is closed as it is the non-strict sublevel set of a continuous function (Theorem \ref{theorem: background: HJB equation for BRS}). If $x\mapsto \overline{v}(\Bar{t},x)$ is of class $\cZ$, then, by Lemma \ref{lemma: approx reach: closure and interior of sublevel sets}, $\text{cl}(\{x\in\R^n\,|\,\overline{v}(\Bar{t},x) < 0\}) =\underline{\mathcal{G}}(\Bar{t})$. Finally, the closure of sets preserves subset relationships, i.e., $\{x\in\R^n\,|\,\overline{v}(\Bar{t},x) < 0\} \subseteq \mathcal{G}(\Bar{t}) $ implies $ \text{cl}(\{x\in\R^n\,|\,\overline{v}(\Bar{t},x) < 0\})\subseteq \text{cl}(\mathcal{G}(\Bar{t}))$, thus $\underline{\mathcal{G}}(\Bar{t}) \subseteq \mathcal{G}(\Bar{t})$. Theorem \ref{theorem: approx reach: HJB equation for under-approximate BRS} then holds by noting that $\Bar{t}\in[0,T]$ is arbitrary.      
\end{proof}
\begin{remark} The condition $\overline{g}(x) > 0$ for all $x\notin \mathcal{X}$ is equivalent to saying that the zero sublevel set of $\overline{g}$, i.e. $\overline{g}^{-1}((-\infty ,0])$, is a subset of $\mathcal{X}$. This is a simple generalisation that allows for $\mathcal{X}$ to be under-approximated by another (potentially more useful) set. For instance, in subsequent sections we will consider the special case where the terminal set is ellipsoidal.
\end{remark}

The backwards reachable set $\mathcal{G}(t)$ has the useful property that for $t_1,t_2\in[0,T]$ with $t_1 \leq t_2$ and for any $x\in\mathcal{G}(t_1)$, there exists a control $u\in\cU[t_1,t_2]$ such that $\varphi(t_2;t_1,x,u)\in\mathcal{G}(t_2)$. This recursive feasibility property is not necessarily inherited by any two arbitrary sets that under-approximate $ \mathcal{G}(t_1)$ and $\mathcal{G}(t_2)$. However, under-approximating sets that are produced via viscosity supersolutions of \eqref{eq: approx reach: HJB PDE} do inherit this property, and we demonstrate this below.
\begin{corollary} Consider the statements of Theorem \ref{theorem: approx reach: HJB equation for under-approximate BRS} and let all of its assumptions hold. Let $t_1,t_2\in[0,T]$ with $t_1 \leq t_2$. Then, for any $\Bar{x} \in  \underline{\mathcal{G}}(t_1)$, there exists a $\Bar{u}\in \cU[t_1,t_2]$ such that $\varphi(t_2; t_1, \Bar{x}, \Bar{u}) \in \underline{\mathcal{G}}(t_2)$.
\label{corollary: approx reach: recursive feasibility for under-approximate BRS}
\end{corollary}
\begin{proof}
Since $\overline{v}$ is a viscosity supersolution of \eqref{eq: approx reach: HJB PDE}, a restriction of $\overline{v}$ to the domain $[0, t_2]\times\R^n$ must also yield a viscosity supersolution of \eqref{eq: approx reach: HJB PDE} on the domain $\Omega = \{(t,x)\in(0,t_2)\times\R^n\,|\,\overline{v}(t,x) < 0\}$. Replacing $T$ with $t_2$ and using the set $\underline{\mathcal{G}}(t_2) = \{x\in\R^n\,|\, \overline{v}(t_2, x) \leq 0\}$ to replace the role of the terminal set $\mathcal{X}$ in Theorem \ref{theorem: approx reach: HJB equation for under-approximate BRS}, it follows that $\underline{\mathcal{G}}(t_1)$ is an under-approximation of the backwards reachable set at time $t=t_1$ from the terminal set $\underline{\mathcal{G}}(t_2)$. 
\end{proof}

\subsection{Over-approximating Backwards Reachable Sets}
Next, let us demonstrate analogous results to Lemma \ref{lemma: approx reach: global extension for local supersolutions}, Theorem \ref{theorem: approx reach: HJB equation for under-approximate BRS}, and Corollary \ref{corollary: approx reach: recursive feasibility for under-approximate BRS} for over-approximating backwards reachable sets. In particular, we show that viscosity subsolutions $\underline{v}$ of \eqref{eq: approx reach: HJB PDE} defined over its zero \emph{superlevel} set can be used to characterise over-approximations of $\mathcal{G}(t)$. We begin by showing that a truncation of $\underline{v}$ from below by 0, that is,
\begin{equation}
    \underline{v}_{\times}(t,x) \doteq \max\{\underline{v}(t,x),\, 0\}, \quad \forall (t,x)\in[0,T]\times\R^n,
    \label{eq: approx reach: global extension subsolution truncated value function}
\end{equation}
yields a viscosity subsolution of \eqref{eq: approx reach: HJB PDE} with $\Omega = (0,T)\times\R^n$. 
\begin{lemma} Let $\underline{v}\in\cC([0,T]\times\R^n;\,\R)$ be a viscosity subsolution of \eqref{eq: approx reach: HJB PDE} with $\Omega = \underline{v}^{-1}((0,\infty))$. Then, $\underline{v}_{\times}\in \cC([0,T]\times\R^n;\,\R)$ as defined in \eqref{eq: approx reach: global extension subsolution truncated value function} is a viscosity subsolution of \eqref{eq: approx reach: HJB PDE} with $\Omega = (0,T)\times\R^n$ and the terminal condition $\underline{v}_{\times}(T,x) = \max\{\underline{v}(T,x),\, 0\}$ for all $x\in \R^n$.
    \label{lemma: approx reach: global extension for local subsolutions}
\end{lemma}
\begin{proof} The proof is largely the same as Lemma \ref{lemma: approx reach: global extension for local supersolutions} so some steps are omitted. Let $\xi\in\cC^1((0,T)\times\R^n;\,\R)$ be defined such that $\underline{v}_{\times} - \xi$ attains a local maximum at a point $(t_0, x_0) \in (0,T)\times\R^n$. As $\underline{v}_{\times}(t,x) \geq 0$ for all $(t,x)\in(0,T)\times\R^n$, either: case \emph{i)} $(t_0, x_0) \in \mathscr{F}^{\underline{v}}_{>0}\doteq (\underline{v}_{\times})^{-1}((0, \infty))$ or case \emph{ii)} $(t_0, x_0) \in\mathscr{F}^{\underline{v}}_{0} \doteq (\underline{v}_{\times})^{-1}(\{0\})$.

\underline{Case \emph{i)}}. Let $(t_0, x_0)\in\mathscr{F}^{\underline{v}}_{>0}$. Then, $\underline{v}_{\times} - \xi$ attains a local maximum within an open set. Since there exists an $r > 0$ such that $\underline{v}(t, x) = \underline{v}_{\times}(t, x)$ for all $(t, x) \in \B_r\left((t_0,x_0)\right)$, $\underline{v} - \xi$ also attains a local maximum at $(t_0, x_0)$. Moreover, as $\underline{v}$ is a viscosity subsolution of \eqref{eq: approx reach: HJB PDE},
    \begin{equation}
        -\xi_t(t_0,x_0) + H(t_0,x_0,\nabla\xi(t_0,x_0)) \leq 0.
        \label{eq: approx reach: global extension subsolution proof 1}
    \end{equation}
    
   \underline{Case \emph{ii)}}. Let $(t_0,x_0)\in\mathscr{F}^{\underline{v}}_{0}$. Note that for all $(t,x)\in(0,T)\times\R^n$, $ \underline{v}_{\times}(t,x) - \underline{v}_{\times}(t_0,x_0)\geq 0$ as $\underline{v}_{\times}(t,x) \geq 0$ by \eqref{eq: approx reach: global extension subsolution truncated value function} and $\underline{v}_{\times}(t_0,x_0) = 0$.  Thus,
    \begin{align*}
        &\liminf_{(t,x)\rightarrow(t_0,x_0)}{ \frac{\underline{v}_{\times}(t,x) - \underline{v}_{\times}(t_0,x_0) - \iprod{(0,0)}{(t-t_0,x-x_0)}}{\norm{(t,x)-(t_0,x_0)}}} \geq 0,
    \end{align*} 
    which implies $0\in D^-\underline{v}_{\times}(t_0,x_0)$ (see Definition \ref{def: background: reg subdifferentials and superdifferentials}). Then, by Lemma \ref{lemma: background: variational description of sub and superdifferentials}, $\left(\xi_t(t_0,x_0), \nabla\xi(t_0,x_0)\right) \in D^+\underline{v}_{\times}(t_0,x_0)$ so $\left(\xi_t(t_0,x_0), \nabla\xi(t_0,x_0)\right) = 0$, which follows using \cite[Proposition 3.1.5, p.51]{CS:04}. This then yields \eqref{eq: approx reach: global extension subsolution proof 1} with equality. From \eqref{eq: approx reach: global extension subsolution proof 1} and Definition \ref{def: background: visc supersolution and subsolution}, $\underline{v}_{\times}$ is a viscosity subsolution of \eqref{eq: approx reach: HJB PDE}. 
\end{proof}
\begin{theorem} Consider the terminal set $\mathcal{X}$ defined in \eqref{eq: background: closed terminal set}. Let $f$ in \eqref{eq: background: nonlinear system} satisfy Assumption \ref{assumption: background: flow field conditions} and let $\underline{v}\in\cC([0,T]\times\R^n;\,\R)$ be a viscosity subsolution of \eqref{eq: approx reach: HJB PDE} with $\Omega = \underline{v}^{-1}((0,\infty))$ and the terminal condition $\underline{v}(T,x)\doteq \underline{g}(x)$ for all $ x\in\R^n$. If $\underline{g}(x) \leq 0$ for all $ x\in \mathcal{X}$, then, the set
\begin{equation}
    \overline{\mathcal{G}}(t) \doteq \left\{x\in\mathbb{R}^n \; | \; \underline{v}(t, x) \leq 0\right\}
    \label{eq: approx reach: closed over approx BRS via HJB zero sublevel set}%
\end{equation}
is an over-approximation of $\mathcal{G}(t)$ for \eqref{eq: background: nonlinear system} for all $t\in[0,T]$.
\label{theorem: approx reach: HJB equation for over-approximate BRS}
\end{theorem}
\begin{proof}
 By Lemma \ref{lemma: approx reach: global extension for local subsolutions}, $\underline{v}_{\times}$ as defined in \eqref{eq: approx reach: global extension subsolution truncated value function} is a viscosity subsolution of
\begin{equation}
    \begin{split}
        -v_t + H(t,x,\nabla v) &= 0, \quad \forall (t,x)\in (0,T)\times\R^n, \\
        v(T,x) = \max\{\underline{g}(x),\,0\} &\doteq \underline{g}_{\times}(x), \quad \forall x\in\R^n.
    \end{split}
    \label{eq: approx reach: HJB PDE over approx BRS proof 1}
\end{equation}
From Theorem \ref{theorem: background: HJB equation for BRS}, $\hat{v}\in\cC([0,T]\times\R^n;\,\R)$ defined by $\hat{v}(t,x) \doteq \inf_{u\in\cU[t,T]}\underline{g}_{\times}(\varphi(T;t,x,u))$ for all $(t,x)\in[0,T]\times\R^n$ is the unique viscosity solution of \eqref{eq: approx reach: HJB PDE over approx BRS proof 1}, so it is also a viscosity supersolution of \eqref{eq: approx reach: HJB PDE over approx BRS proof 1}. By Theorem \ref{theorem: background: global comparison result}, 
\begin{equation}
    \underline{v}_{\times}(t,x) \leq \hat{v}(t,x), \quad \forall(t,x)\in[0,T]\times\R^n.
   \label{eq: approx reach: HJB PDE over approx BRS proof 2}
\end{equation}

Next, fix any $\Bar{t}\in[0,T]$ and take any $\Bar{x} \in \{x\in\R^n\,|\,\underline{v}(\Bar{t},x) > 0\}$ so that $\underline{v}_{\times}(\Bar{t},\Bar{x}) >0$. Then, from \eqref{eq: approx reach: HJB PDE over approx BRS proof 2},  $\hat{v}(\Bar{t},\Bar{x}) \geq  \underline{v}_{\times}(\Bar{t},\Bar{x}) \doteq \delta_{\Bar{x}} > 0$, which yields
\begin{equation}
   \hat{v}(\Bar{t},\Bar{x})\doteq \inf_{u\in\cU[\Bar{t},T]}\max\{\underline{g}(\varphi(T;\Bar{t},\Bar{x},u)),\,0\} \geq \delta_{\Bar{x}}.
   \label{eq: approx reach: HJB PDE over approx BRS proof 3}
\end{equation}
Since $0$ is independent of the control $u$, \eqref{eq: approx reach: HJB PDE over approx BRS proof 3} implies that $\max\{\inf_{u\in\cU[\Bar{t},T]}\underline{g}(\varphi(T;\Bar{t},\Bar{x},u)),\,0\} \geq \delta_{\Bar{x}}$, and so $\inf_{u\in\cU[\Bar{t},T]}\underline{g}(\varphi(T;\Bar{t},\Bar{x},u))\geq \delta_{\Bar{x}}$. For any $\hat{u}\in\cU[\Bar{t},T]$,
\begin{equation}
    \underline{g}(\varphi(T;\Bar{t},\Bar{x},\hat{u})) \geq \inf_{u\in\cU[\Bar{t},T]}\underline{g}(\varphi(T;\Bar{t},\Bar{x},u)) \geq \delta_{\Bar{x}} > 0.
    \label{eq: approx reach: HJB PDE over approx BRS proof 4}
\end{equation}

The statement $\underline{g}(x) \leq 0$ for all $ x \in \mathcal{X}$ is equivalent to the statement $\underline{g}(x) > 0 $ implies $ x \notin \mathcal{X}$. Then, \eqref{eq: approx reach: HJB PDE over approx BRS proof 4} implies that $\mathcal{X}$ can not be reached from $\Bar{x}$ under any admissible control. Thus, $\Bar{x} \in \left(\mathcal{G}(\Bar{t})\right)^c$. Finally, as $\Bar{t}\in[0,T]$ and $\Bar{x}\in\{x\in\R^n\,|\,\underline{v}(\Bar{t},x) > 0\}$ are arbitrary, it follows that $\{x\in\R^n\,|\,\underline{v}(t,x) > 0\} \subseteq \left(\mathcal{G}(t)\right)^c$ for all $t\in[0,T]$. Taking the set complement yields $\mathcal{G}(t) \subseteq \{x\in\R^n\,|\,\underline{v}(t,x) \leq 0\} \doteq \overline{\mathcal{G}}(t)$. 
\end{proof}

An analogous property to recursive feasibility is a recursive infeasibility property. In particular, if a state $x\notin \mathcal{G}(t_1)$, then for all future times $t_2 \in [t_1,T]$, it is impossible for the system to be driven into $\mathcal{G}(t_2)$ under any control $u\in\cU[t_1,t_2]$. Consequently, being outside of the backwards reachable set provides a certificate of safety, that the (potentially) unsafe terminal set $\mathcal{X}$ will not be reached. This recursive infeasibility property is inherited by over-approximating sets produced via viscosity subsolutions of \eqref{eq: approx reach: HJB PDE}, which we demonstrate below.
\begin{corollary} Consider the statements of Theorem \ref{theorem: approx reach: HJB equation for over-approximate BRS} and let all of its assumptions hold. Let $t_1,t_2\in[0,T]$ with $t_1 \leq t_2$. Then, for any $\Bar{x} \notin \overline{\mathcal{G}}(t_1)$ and for any $u\in\cU[t_1,t_2]$, $\varphi(t_2; t_1, \Bar{x}, u) \notin \overline{\mathcal{G}}(t_2)$.
\label{corollary: approx reach: recursive infeasibility for over-approximate BRS}
\end{corollary}
\begin{proof}
First note that a restriction of $\underline{v}$ to the domain $[0, t_2]\times\R^n$ must also yield a viscosity subsolution of \eqref{eq: approx reach: HJB PDE} on the domain $\Omega = \{(t,x)\in(0,t_2)\times\R^n\,|\,\underline{v}(t,x) > 0\}$. Replacing $T$ with $t_2$ in Theorem \ref{theorem: approx reach: HJB equation for over-approximate BRS}, it follows that $\overline{\mathcal{G}}(t_1)$ is an over-approximation of the backwards reachable set at time $t = t_1$ from the terminal set $\overline{\mathcal{G}}(t_2) = \left\{x\in\mathbb{R}^n \; | \; \underline{v}(t_2, x) \leq 0\right\}$, which is denoted by $\mathcal{G}_v(t_1)$. If $\Bar{x} \notin \overline{\mathcal{G}}(t_1) \doteq \{x\in\R^n\,|\,\underline{v}(t_1, x) \leq 0\}$, then there can not exist a control $u\in\cU[t_1,t_2]$ such that $\varphi(t_2; t_1, \Bar{x}, u)\in \overline{\mathcal{G}}(t_2)$, else it would be true that $\Bar{x}\in\mathcal{G}_v(t_1)$, which can not hold since $\mathcal{G}_v(t_1) \subseteq  \overline{\mathcal{G}}(t_1)$. 
\end{proof}

\subsection{Generating Viscosity Supersolutions and Subsolutions}
Before we use the previous findings to develop ellipsoidal approximation schemes for LTV systems, we discuss a possible method of constructing viscosity supersolutions and subsolutions. Consider a collection of $\cC^1([0,T]\times\R^n;\,\R)$ functions $\{\overline{v}_i\}^{n_q}_{i=1}$ and $\{\underline{v_i}\}^{n_q}_{i=1}$. Suppose that every $\overline{v}_i$ and $\underline{v_i}$ satisfies the respective partial differential inequalities (PDIs)
\begin{align}
    -v_t + H(t, x, \nabla v) &\geq 0, \quad \forall (t, x) \in \overline{v}_i^{-1}((-\infty,0)),     \label{eq: approx reach: HJB PDE under approx BRS collection}\\
    -v_t + H(t, x, \nabla v) &\leq 0, \quad \forall (t, x) \in \underline{v_i}^{-1}((0,\infty)), \label{eq: approx reach: HJB PDE over approx BRS collection}
\end{align}
with $H:[0,T]\times\R^n\times\R^n\rightarrow\R$ given by \eqref{eq: background: HJB Hamiltonian}. Then, every $\overline{v}_i$ and $\underline{v_i}$ is, respectively, a classical supersolution and subsolution of \eqref{eq: approx reach: HJB PDE} over the domains $\Omega = \overline{v}_i^{-1}((-\infty,0))$ and $\Omega = \underline{v_i}^{-1}((0,\infty))$. It then follows that the minimisation 
over $\{\overline{v}_i\}^{n_q}_{i=1}$ and the maximisation over $\{\underline{v_i}\}^{n_q}_{i=1}$, 
\begin{align}
    v_m(t,x) &\doteq \min_{1 \leq i \leq n_q}\overline{v}_i(t,x),\quad \forall (t,x)\in[0,T]\times\R^n,
    \label{eq: approx reach: min collection function def}\\
    v_M(t,x) &\doteq \max_{1 \leq i \leq n_q}\underline{v_i}(t,x),\quad \forall (t,x)\in[0,T]\times\R^n,
    \label{eq: approx reach: max collection function def}
\end{align}
produce a viscosity supersolution and subsolution of \eqref{eq: approx reach: HJB PDE}, respectively. This is demonstrated in the following result, which can be seen as an application of \cite[Proposition 2.1, p.34]{bardi1997optimal} to the problem setting of this work. 
\begin{lemma}
    Let $\{\overline{v}_i\}^{n_q}_{i=1}$ (resp. $\{\underline{v_i}\}^{n_q}_{i=1}$) be a finite collection of $n_q \in \N$ functions belonging to $\cC^1([0,T]\times\R^n;\,\R)$. Let every $\overline{v}_i$ (resp. $\underline{v_i}$) satisfy the PDI in \eqref{eq: approx reach: HJB PDE under approx BRS collection} (resp. \eqref{eq: approx reach: HJB PDE over approx BRS collection}). Then, $v_m$ as defined in \eqref{eq: approx reach: min collection function def} (resp. $v_M$ as defined in \eqref{eq: approx reach: max collection function def}) is a viscosity supersolution (resp. subsolution) of \eqref{eq: approx reach: HJB PDE} with $\Omega = v_m^{-1}((-\infty,0))$ (resp. $\Omega = v_M^{-1}((0,\infty))$).
\label{lemma: approx reach: min and max function over supersolutions and subsolutions}
\end{lemma}
\begin{proof}
   To demonstrate that $v_m$ is a viscosity supersolution of \eqref{eq: approx reach: HJB PDE} with $\Omega = v_m^{-1}((-\infty,0))$, let $\xi\in\cC^1((0,T)\times\R^n;\,\R)$ be defined such that $v_m-\xi$ attains a local minimum at $(t_0,x_0)\in v_m^{-1}((-\infty,0))$. Since  $v_m^{-1}((-\infty,0))$ is an open set, for some $r>0$ the following holds for all $(t,x)\in\B_r((t_0,x_0))$
    \begin{equation}
        v_m(t,x) - \xi(t,x) \geq v_m(t_0,x_0) - \xi(t_0,x_0).
        \label{eq: approx reach: min function over supersolutions proof 1}
    \end{equation}   
    Let $\hat{n} \in \{1,\ldots, n_q\}$ be any index such that $\overline{v}_{\hat{n}}(t_0,x_0) = v_m(t_0,x_0)$. By Definition, $v_m(t,x) \leq \overline{v}_i(t,x)$ for all $(t,x)\in(0,T)\times\R^n$ and for all $i \in \{1,\ldots, n_q\}$. Thus, from \eqref{eq: approx reach: min function over supersolutions proof 1},
    \begin{equation}
        \overline{v}_{\hat{n}}(t,x) - \xi(t,x) \geq \overline{v}_{\hat{n}}(t_0,x_0) - \xi(t_0,x_0),
    \end{equation}
    for all $(t,x)\in\B_r((t_0,x_0))$, which means $\overline{v}_{\hat{n}} - \xi$ also attains a local minimum at $(t_0,x_0)$. Since $\overline{v}_{\hat{n}}$ is continuously differentiable, which implies that $\left(\frac{\partial }{\partial t}\overline{v}_{\hat{n}}(t_0, x_0), \nabla \overline{v}_{\hat{n}}(t_0, x_0)\right) = \left(\xi_t(t_0,x_0),  \nabla \xi(t_0,x_0)\right)$. Moreover, $(t_0,x_0) \in \overline{v}_{\hat{n}}^{-1}((-\infty,0))$ since $\overline{v}_{\hat{n}}(t_0,x_0) = v_m(t_0,x_0) < 0$. It then follows from \eqref{eq: approx reach: HJB PDE under approx BRS collection} that $-\xi_t(t_0,x_0) + H(t_0, x_0, \nabla \xi(t_0,x_0)) \geq 0$. As $(t_0, x_0)\in v_m^{-1}((-\infty,0))$ is arbitrary, by Definition \ref{def: background: visc supersolution and subsolution}, $v_m$ is a viscosity supersolution of \eqref{eq: approx reach: HJB PDE} on $\Omega = v^{-1}_m((-\infty,0))$. 
    
With appropriate sign changes, the analogous result for $v_M$ can be demonstrated. Briefly, consider a function $\xi\in\cC^1((0,T)\times\R^n;\,\R)$, which is defined such that $v_M-\xi$ attains a local maximum at $(t_0,x_0)\in v^{-1}_M((0,\infty))$. Since $v_M(t,x) \geq \underline{v_i}(t,x)$ for all $(t,x)\in[0,T]\times\R^n$ and for all $i \in \{1,\ldots,n_q\}$, this implies that for some $\hat{n}\in\{1,\ldots,n_q\}$ with $\underline{v_{\hat{n}}}(t_0,x_0)= v_M(t_0,x_0)$, $\underline{v_{\hat{n}}} - \xi$ also attains a local maximum at $(t_0,x_0)$. Thus, $\left(\frac{\partial }{\partial t}\underline{v_{\hat{n}}}(t_0, x_0), \nabla \underline{v_{\hat{n}}}(t_0, x_0)\right) = \left(\xi_t(t_0,x_0),  \nabla \xi(t_0,x_0)\right)$ and by \eqref{eq: approx reach: HJB PDE over approx BRS collection}, $-\xi_t(t_0,x_0) + H(t_0, x_0, \nabla \xi(t_0,x_0)) \leq 0$. So, $v_M$ is a viscosity subsolution of \eqref{eq: approx reach: HJB PDE} with $\Omega = v^{-1}_M((0,\infty))$. 
\end{proof}

\begin{remark}
    The use of viscosity solutions as opposed to classical solutions (as in \cite{LMD23}) to characterise under- and over-approximations of reachable sets may seem unnecessary in the case where each individual function $\overline{v}_i$ and $\underline{v_i}$ is a classical supersolution and subsolution of \eqref{eq: approx reach: HJB PDE}. However, the results of this section generalise to other methods of constructing non-classical solutions and, moreover, provide a unified and rigorous means for understanding how collections of sets can be used in reachability-based problems.
\end{remark}

%% file: text/LTV_reachability.tex
\section{Approximate Reachability for LTV Systems}
\label{sec: LTV reachability}
Grid-based approaches are a conventional means of producing accurate characterisations of reachable sets for a broad class of systems and problem settings. However, these approaches scale poorly with the number of system states even when dealing with linear systems. In this section, we exploit the structure of linear systems via the results of Section \ref{sec: approximate reachability} to yield ellipsoidal approximation schemes for LTV systems. The computational cost of these schemes is dominated by the integration of a matrix differential equation, which has a computational complexity of $\mathcal{O}(n^3)$. 
\begin{definition}
Let $q \in \mathbb{R}^n$ and $Q\in \spd{n}$. An ellipsoid with centre $q$ and shape $Q$ is defined as the set
\begin{equation}
    \mathcal{E}(q, Q) \doteq \left\{x\in\mathbb{R}^n\, | \, \iprod{x - q}{Q^{-1}(x-q)} \leq 1\right\}. 
    \label{eq: ellipsoidal approx: ellipsoidal set def}
\end{equation}
\vspace{-1em}
\label{def: ellipsoidal approx: ellipsoidal set}
\end{definition}
\begin{remark}
    An ellipsoid $\mathcal{E}(q,Q)$ has its axes aligned along the eigenvectors of $Q$ with lengths equal to the square root of its corresponding eigenvalues. An ellipsoid is equivalently defined as the affine transformation of the unit ball, in particular,
\begin{equation}
    \mathcal{E}(q,Q) = \left\{Q^{\frac{1}{2}}w + q\in\mathbb{R}^n\, |\,\norml{w}^2 \leq 1\right\}.
    \label{eq: ellipsoidal approx: ellipsoidal set at affine transformation}
\end{equation}
\vspace{-1em}
\label{remark: ellipsoidal approx: ellipsoid is affine transformation}
\end{remark}
Let us now discuss the problem set up that will be used for the remainder of this paper. We consider a special case of \eqref{eq: background: nonlinear system} where $f:[0,T]\times\R^n\times\U\rightarrow \R^n$ now describes a linear time-varying system subject to an ellipsoidal input set $\U$ and an ellipsoidal terminal state set $\mathcal{X}\subset \mathbb{R}^n$. In particular, we consider the continuous-time system given by
\begin{equation}
    \begin{split}
    &\Dot{x}(s) = A(s)x(s) + B(s)u(s),\quad \text{a.e. } s\in(t, T), \\
    &u(s)\in\U = \mathcal{E}(p, P),\quad x(T) \in \mathcal{X} = \mathcal{E}(x_e, X_e),
    \end{split}
\label{eq: ellipsoidal approx: linear system}
\end{equation}
where $A\in\cC([0,T];\,\R^{n\times n})$, $B\in\cC([0,T];\,\R^{n\times m})$, $p\in\R^n$, $P\in\spd{m}$, $x_e \in \R^n$, and $X_e \in \spd{n}$. In the notation of \eqref{eq: background: closed terminal set}, the terminal state set $\mathcal{X}$ can be defined via $g(x) = \iprod{x - x_e}{X_e^{-1}(x - x_e)} - 1$. 

Note that this problem setting satisfies the assumptions imposed on the nonlinear system considered in previous sections. To see this, observe that continuity of $A(\cdot)$ and $B(\cdot)$ ensures item \emph{\ref{item: background: flow field conditions item 1})} of Assumption \ref{assumption: background: flow field conditions} is satisfied. For any $s\in[0,T]$ and $u\in\U$, the map $x\mapsto A(s)x+B(s)u$ is affine, which ensures the Lipschitz continuity assumption is satisfied, and convexity of $\U$ then ensures that Assumption \ref{assumption: background: convex flow field} is satisfied.

Our proposed approximation scheme utilises a collection of ellipsoidal sets $\left\{\mathcal{E}\left(q_i(t), Q_i(t)\right)\right\}^{n_q}_{i=1}$ and defines a corresponding evolution for each center $q_i:[0,T]\rightarrow \mathbb{R}^n$ and shape $Q_i:[0,T]\rightarrow\spd{n}$ such that every ellipsoid is an under-approximation (resp. over-approximation) of $\mathcal{G}(t)$ for \eqref{eq: ellipsoidal approx: linear system}. The union of the under-approximating ellipsoids (resp. intersection of the over-approximating ellipsoids) remains an under-approximation (resp. over-approximation) of $\mathcal{G}(t)$. It can be readily verified that the union of ellipsoids can be characterised by the zero sublevel set of the minimum over a collection of quadratic functions. In particular, 
\begin{equation}
    \bigcup^{n_q}_{i=1}\mathcal{E}(q_i(t),Q_i(t)) = \{x\in\R^n\,|\, \overline{e}(t, x) \leq 0\}, \label{eq: ellipsoidal approx: min over quadratics is union}
\end{equation}
where, for all $(t,x)\in[0,T]\times\R^n$,
\begin{equation}
    \overline{e}(t,x) \doteq \min_{1 \leq i \leq n_q}\iprod{x-q_i(t)}{Q^{-1}_i(t)(x-q_i(t))}- 1.
    \label{eq: ellipsoidal approx: proposed under-approx value function}
\end{equation}
When the evolution of all pairs $\{q_i, Q_i\}$ are chosen so that the PDI in \eqref{eq: approx reach: HJB PDE under approx BRS collection} is satisfied for every quadratic function, \eqref{eq: ellipsoidal approx: proposed under-approx value function} will construct a viscosity supersolution in the manner suggested by Lemma \ref{lemma: approx reach: min and max function over supersolutions and subsolutions}. Correspondingly, the intersection of ellipsoids can be characterised by the zero sublevel set of the maximum over a collection of quadratics, that is,
\begin{equation}
     \bigcap^{n_q}_{i=1}\mathcal{E}(q_i(t),Q_i(t)) = \{x\in\R^n\,|\, \underline{e}(t, x) \leq 0\}, \label{eq: ellipsoidal approx: max over quadratics is intersect}
\end{equation}
where, for all $(t,x)\in[0,T]\times\R^n$, 
\begin{equation}
    \underline{e}(t,x) \doteq  \max_{1 \leq i \leq n_q}\iprod{x-q_i(t)}{Q^{-1}_i(t)(x-q_i(t))}- 1.
    \label{eq: ellipsoidal approx: proposed over-approx value function}
\end{equation}
This will construct a viscosity subsolution in the manner suggested by Lemma \ref{lemma: approx reach: min and max function over supersolutions and subsolutions}. Note that the map $x\mapsto\overline{e}(t,x)$ is of class $\cZ$ (see Definition \ref{def: approx reach: zero regular continuous function}), which is needed in the assumptions of Theorem \ref{theorem: approx reach: HJB equation for under-approximate BRS}. We state this below with the proof postponed to the Appendix.
\begin{lemma} Consider $\overline{e}\in\cC([0,T]\times\R^n;\,\R)$ as defined in \eqref{eq: ellipsoidal approx: proposed under-approx value function} with $Q_i(t) \in \spd{n}$ for all $t\in[0,T]$. Then, the map $x\mapsto \overline{e}(t,x)$ is of class $\cZ$ for all $t\in[0,T]$.
    \label{lemma: ellipsoidal approx: min over quadratics attached sublevel set}
\end{lemma}

Let us present two preliminary results that will be needed to ensure that $\overline{e}$ and $\underline{e}$ are, respectively, viscosity supersolutions and subsolutions of \eqref{eq: approx reach: HJB PDE} over an appropriate domain.
\begin{lemma}
Let $Q\in \spsd{n}$ and $S \in \mathbb{O}^n$ be an orthogonal matrix. Then, for any $w\in\R^n$ with $\norml{w} \leq 1$,
\begin{equation}
    \big\lVert Q^\half w\big\rVert_2 \geq \tfrac{1}{2}\iprod{w}{(Q^{\frac{1}{2}}S+S^TQ^{\frac{1}{2}})w}.
    \label{eq: ellipsoidal approx: lower bound on norm over transformed unit ball}
\end{equation}
Moreover, equality in \eqref{eq: ellipsoidal approx: lower bound on norm over transformed unit ball} is attained at $w^\star\in\R^n$ if $\norml{w^\star} = 1$ and $Sw^\star = Q^\half w^\star\slash \lVert Q^\half w^\star \rVert_2$.
\label{lemma: ellipsoidal approx: lower bound on norm over transformed unit ball}
\end{lemma}
\begin{proof}
Fix arbitrary $Q \in \spsd{n}$, $S \in \mathbb{O}^n$, and $w\in\R^n$ such that $\norml{w} \leq 1$. Noting that $Q^\half$ exists, the Cauchy-Schwarz inequality implies that
\begin{align*}
    \tfrac{1}{2}\iprod{w}{(Q^{\frac{1}{2}}S+S^TQ^{\frac{1}{2}})w} &= \iprod{Q^\half w}{Sw}
    \leq \bnorml{Q^\half w}\bnorml{Sw} \\
    &=  \bnorml{Q^\half w}\norml{w} \leq  \bnorml{Q^\half w},
\end{align*}
which is \eqref{eq: ellipsoidal approx: lower bound on norm over transformed unit ball}. For the second assertion, let $w^\star \in \R^n$ be such that $\norml{w^\star} = 1$ and $S w^\star = Q^\half w^\star\slash \lVert Q^\half w^\star \rVert_2$. Substituting $w^\star$ into the right-hand side of \eqref{eq: ellipsoidal approx: lower bound on norm over transformed unit ball},
\begin{align*}
    &\tfrac{1}{2}\iprod{w^\star}{(Q^{\frac{1}{2}}S+S^TQ^{\frac{1}{2}})w^\star} = \iprod{Q^\half w^\star}{Sw^\star} =\iprod{Q^\half w^\star} {Q^\half w^\star \slash \bnorml{Q^\half w^\star}} = \bnorml{Q^\half w^\star}. 
\end{align*}
\end{proof}
\begin{lemma}
Let $Q\in \spsd{n}$ and $\kappa \in \R_{>0}$. Then, for any $w\in\R^n$ with $\norml{w} \geq 1$,
\begin{equation}
    \bnorml{Q^\half w} \leq\tfrac{1}{2}\iprod{w}{\left(\tfrac{1}{\kappa}Q + \kappa\I\right)w}.
    \label{eq: ellipsoidal approx: upper bound on norm over transformed unit ball}
\end{equation}
Moreover, equality in \eqref{eq: ellipsoidal approx: upper bound on norm over transformed unit ball} is attained at $w^\star \in \R^n$ if $\norml{w^\star} = 1$ and $\kappa = \bnorml{Q^\half w^\star} \neq 0$.
\label{lemma: ellipsoidal approx: upper bound on norm over transformed unit ball}
\end{lemma}
\begin{proof}
Fix arbitrary $Q\in\spsd{n}$, $\kappa \in \R_{>0}$, and $w\in\R^n$ such that $\norml{w} \geq 1$. Then,
\begin{align}
    &\tfrac{1}{2}\iprod{w}{\left(\tfrac{1}{\kappa}Q + \kappa\I\right)w} - \bnorml{Q^\half w} \nn \\
    &\qquad = \tfrac{1}{2\kappa}\left(\bnorml{Q^\half w}^2- 2\kappa\bnorml{Q^\half w} + \kappa^2\norml{w}^2\right) \label{eq: ellipsoidal approx: upper bound on norm proof 1}\\
    &\qquad \geq \tfrac{1}{2\kappa}\left(\bnorml{Q^\half w}^2- 2\kappa\bnorml{Q^\half w} + \kappa^2\right)\nn \\
    &\qquad= \tfrac{1}{2\kappa}\left(\bnorml{Q^\half w}- \kappa\right)^2 \geq 0, \label{eq: ellipsoidal approx: upper bound on norm proof 2}
\end{align}
which gives \eqref{eq: ellipsoidal approx: upper bound on norm over transformed unit ball}. For the second assertion, let $w^\star \in \R^n$ be such that $\norml{w^\star} = 1$ and $\bnorml{Q^\half w^\star} \neq 0$. Substitution of $\norml{w^\star} = 1$ in \eqref{eq: ellipsoidal approx: upper bound on norm proof 1} followed by substitution of $\kappa = \bnorml{Q^\half w^\star}$ in \eqref{eq: ellipsoidal approx: upper bound on norm proof 2} yields the desired equality in \eqref{eq: ellipsoidal approx: upper bound on norm over transformed unit ball}. 
\end{proof}

\subsection{Ellipsoidal Under-approximation of Backwards Reachable Sets}
We can now present our main result on under-approximating the backwards reachable set of \eqref{eq: ellipsoidal approx: linear system} via a union of ellipsoidal sets. To ensure that each ellipsoid $\mathcal{E}(q_i(t),Q_i(t))$ in \eqref{eq: ellipsoidal approx: min over quadratics is union} is an under-approximation of $\mathcal{G}(t)$, we evolve each pair $\{q_i, Q_i\}$ backwards in time according to the final value problem (FVP) defined via \eqref{eq: ellipsoidal approx: linear system} by 
\begin{equation}
\left\{\;
\begin{split}
    \dot{q}_i(t) &= A(t)q_i(t) + B(t)p,\\
    \dot{Q}_i(t) &= A(t)Q_i(t)+Q_i(t)A^T(t) -\underline{\calR_i}(t, Q_i(t)), \\
    q_i(T) &= x_e, \quad Q_i(T) = X_e,
\end{split} \right.
    \label{eq: ellipsoidal approx: under approx BRS dynamics + terminal condition}
\end{equation}
for all $i\in\{1,\ldots,n_q\}$, a.e. $t\in (0,T)$, in which the parameterised operators $\underline{\calR_i}:[0,T]\times \spd{n}\rightarrow \sym{n}$ and  $\calQ_{\star}:[0,T]\times \spd{n}\rightarrow \spsd{n}$ are defined by
\begin{align}
    &\underline{\calR_i}(t, Q) \doteq Q^\half \left(\calQ^{\frac{1}{2}}_{\star}(t, Q)S_i(t) + S^T_i(t)\calQ^{\frac{1}{2}}_{\star}(t, Q)\right)Q^\half,
    \label{eq: ellipsoidal approx: under approx BRS Qu matrix}\\
    &\calQ_{\star}(t, Q) \doteq Q^{-\half}B(t)PB^T(t)Q^{-\half}, \label{eq: ellipsoidal approx: under approx BRS Q star matrix}
\end{align} \normalsize
with $S_i:[0,T]\rightarrow \mathbb{O}^n$ being continuous. We use the short-hand notation $\calQ_{i,\star}(t) \doteq \calQ_{\star}(t,Q_i(t))$ where appropriate. 
\begin{assumption}
    $S_i\in\cC\left([0,T];\,\mathbb{O}^n\right)$ is selected such that a unique positive definite solution $Q_i:[0,T]\rightarrow \spd{n}$ exists. \label{assumption: ellipsoidal approx: S matrix choice for positive definiteness}
\end{assumption} 
\noindent Existence of solutions for \eqref{eq: ellipsoidal approx: under approx BRS dynamics + terminal condition} and the role of Assumption \ref{assumption: ellipsoidal approx: S matrix choice for positive definiteness} is further discussed in Remark \ref{remark: ellipsoidal approx: positive definiteness of ellipsoid shape} following the statement of the main result. 

The orthogonal matrix-valued function $S_i\in\cC\left([0,T];\,\mathbb{O}^n\right)$ is a degree of freedom that allows each ellipsoid to evolve under a different set of dynamics. Under appropriate conditions, a solution $x^\star(\cdot)$ of \eqref{eq: ellipsoidal approx: linear system} can be contained within an ellipsoid $\mathcal{E}(q_i(t), Q_i(t))$ by selecting $S_i$ to satisfy
\begin{equation}
    \frac{S_i(s)\hat{w}_i(s)}{\big\lVert \hat{w}_i(s)\big\rVert_2} = \frac{\calQ^{\half}_{i,\star}(s)\hat{w}_i(s)}{\big\lVert \calQ^{\half}_{i,\star}(s)\hat{w}_i(s)\big\rVert_2}, 
    \label{eq: ellipsoidal approx: under approx BRS orthogonal matrix}
\end{equation}
 for all $s\in[t,T]$ such that $\calQ^{\half}_{i,\star}(s)\hat{w}_i(s) \neq 0$, in which
 \begin{equation}
     \hat{w}_i(s) \doteq Q^{-\half}_i(s)(x^\star(s)-q_i(s)). \label{eq: ellipsoidal approx: hat w for orthogonal matrix}
 \end{equation}
\begin{theorem} Let Assumption \ref{assumption: ellipsoidal approx: S matrix choice for positive definiteness} hold and consider the LTV system described by \eqref{eq: ellipsoidal approx: linear system}. Let $\{q_i, Q_i\}^{n_q}_{i=1}$ be the set of solutions to \eqref{eq: ellipsoidal approx: under approx BRS dynamics + terminal condition}-\eqref{eq: ellipsoidal approx: under approx BRS Q star matrix}. 
\begin{enumerate}[(a)]
    \item The set
\begin{equation}
    \underline{\mathcal{G}}(t) \doteq \bigcup^{n_q}_{i=1}\mathcal{E}\left(q_i(t),Q_i(t)\right)
    \label{eq: ellipsoidal approx: under-approx BRS as a union of ellipsoids}
\end{equation}
is an under-approximation of $\mathcal{G}(t)$ for \eqref{eq: ellipsoidal approx: linear system} for all $t\in[0,T]$. \label{item: ellipsoidal approx: under-approx BRS assertion 1}
\item Given $S_i\in\cC\left([0,T];\,\mathbb{O}^n\right)$ with $i\in\{1,\ldots,n_q\}$ satisfying \eqref{eq: ellipsoidal approx: under approx BRS orthogonal matrix} for all $s\in[t,T]$ such that $\calQ^{\half}_{i,\star}(s)\hat{w}_i(s) \neq 0$ and $x^\star(\cdot)$ is a solution of \eqref{eq: ellipsoidal approx: linear system} with $x^\star(T)\in \mathcal{X}$, then,
\begin{equation}
    x^\star(t) \in \mathcal{E}(q_i(t), Q_i(t))\subseteq \underline{\mathcal{G}}(t)
    \label{eq: ellipsoidal approx: solution contained in under approx ellipsoid}
\end{equation} for all $t\in[0,T]$.
\label{item: ellipsoidal approx: under-approx BRS assertion 2}
\end{enumerate}
\label{theorem: ellipsoidal approx: under-approx BRS}
\end{theorem}
\begin{proof} \underline{Assertion \emph{(\ref{item: ellipsoidal approx: under-approx BRS assertion 1})}}. Fix any $i \in \{1,\ldots,n_q\}$ and note that by Assumption \ref{assumption: ellipsoidal approx: S matrix choice for positive definiteness}, $Q_i(t)\in\spd{n}$, which implies $Q^{-1}_i(t)$ exists for all $t\in[0,T]$. Define $\overline{e}_i\in\cC^1([0,T]\times\R^n;\,\R)$ by
\begin{equation}
    \overline{e}_i(t,x) \doteq \iprod{x-q_i(t)}{Q_i^{-1}(t)(x-q_i(t))} -1 ,
    \label{eq: ellipsoidal approx: under-approx BRS proof 1}
\end{equation}
where $\{q_i,Q_i\}$ is the solution to \eqref{eq: ellipsoidal approx: under approx BRS dynamics + terminal condition}-\eqref{eq: ellipsoidal approx: under approx BRS Q star matrix}. Since $Q_i(t)\in\spd{n}$ also implies the existence of $Q^{\half}_i(t)$, $\calQ_{i,\star}(t)$ is well-defined. Positive definiteness of $P$ implies $\calQ_{i,\star}(t)\in \spsd{n}$, so $\calQ^\half_{i,\star}(t)\in \spsd{n}$. Differentiating \eqref{eq: ellipsoidal approx: under-approx BRS proof 1} yields
\begin{align}
    \tfrac{\partial}{\partial t}\overline{e}_i &= -2\iprod{\dot{q}_i(t)}{Q_i^{-1}(t)(x-q_i(t))} +\iprod{x - q_i(t)}{\tfrac{d}{d t}Q^{-1}_i(t)(x - q_i(t))}\label{eq: ellipsoidal approx: under-approx BRS proof 2},\\
    \nabla \overline{e}_i &= 2Q^{-1}_i(t)(x - q_i(t)). \label{eq: ellipsoidal approx: under-approx BRS proof 3}
\end{align}
Noting that $\frac{d}{dt}Q^{-1}_i(t) = -Q^{-1}_i(t)\dot{Q}_i(t)Q^{-1}_i(t)$, substitution of \eqref{eq: ellipsoidal approx: under-approx BRS proof 2}-\eqref{eq: ellipsoidal approx: under-approx BRS proof 3} in the left-hand side of \eqref{eq: approx reach: HJB PDE under approx BRS collection} yields
\begin{align}
-\tfrac{\partial}{\partial t}\overline{e}_i + H(t, x, \nabla\overline{e}_i) &= 2\iprod{\dot{q}_i}{Q_i^{-1}(x-q_i)}+\iprod{x - q_i}{Q^{-1}_i\dot{Q}_iQ^{-1}_i(x - q_i)} \nn\\
    &+ 2\max_{u\in\mathcal{E}(p,P)}\iprod{-Q^{-1}_i(x - q_i)}{Ax+Bu},
    \label{eq: ellipsoidal approx: under-approx BRS proof 4}
\end{align}
in which the time dependence is omitted for brevity. Standard results on support functions over ellipsoidal sets (see \cite[Lemma 10]{LMD23}) yield
\begin{align}
    &\max_{u\in\mathcal{E}(p,P)}\iprod{-B^TQ^{-1}_i(x-q_i)}{u} = \nn \\
    &\qquad \qquad  \iprod{-B^TQ^{-1}_i(x-q_i)}{p}  +\sqrt{\iprod{x-q_i}{Q^{-1}_iBPB^TQ^{-1}_i(x-q_i)}}. \label{eq: ellipsoidal approx: under-approx BRS proof 5}
\end{align}
Substituting \eqref{eq: ellipsoidal approx: under-approx BRS proof 5} in \eqref{eq: ellipsoidal approx: under-approx BRS proof 4} yields
\begin{align}
        -\tfrac{\partial}{\partial t}\overline{e}_i + H(t, x, \nabla\overline{e}_i) &= 2\iprod{x - q_i}{Q^{-1}_i(\dot{q}_i - Aq_i - Bp)} \nn\\
          & + \iprod{x - q_i}{\left[Q^{-1}_i\dot{Q}_iQ^{-1}_i - A^TQ^{-1}_i - Q^{-1}_iA\right](x - q_i)}\nn\\
        &  + 2\sqrt{\iprod{x-q_i}{Q^{-1}_iBPB^TQ^{-1}_i(x-q_i)}},
       \label{eq: ellipsoidal approx: under-approx BRS proof 6}
\end{align}
so that by \eqref{eq: ellipsoidal approx: under approx BRS dynamics + terminal condition}-\eqref{eq: ellipsoidal approx: under approx BRS Q star matrix}, 
\begin{align}
        -\tfrac{\partial}{\partial t}\overline{e}_i + H(t, x, \nabla\overline{e}_i) &= 2\sqrt{\iprod{x-q_i}{Q^{-\half}_i \calQ_{i,\star}Q^{-\half}_i(x-q_i)}}\nn \\
   & -\iprod{x - q_i}{Q^{-\half}_i\left(\calQ^{\half}_{i,\star}S_i + S^T_i\calQ^{\half}_{i,\star}\right)Q^{-\half}_i(x - q_i)}.
    \label{eq: ellipsoidal approx: under-approx BRS proof 7}
\end{align} 
To satisfy the PDI in \eqref{eq: approx reach: HJB PDE under approx BRS collection}, the right-hand side of \eqref{eq: ellipsoidal approx: under-approx BRS proof 7} is required to be non-negative for points $(t,x)\in\overline{e}_i^{-1}((-\infty,0))$. Equivalently, with $x = Q^\half_i(t)w + q_i(t)$, the following can instead be considered for all $\norml{w} < 1$ (see Remark \ref{remark: ellipsoidal approx: ellipsoid is affine transformation}):
\begin{equation}
    2\bnorml{\calQ^{\half}_{i,\star}w} - \iprod{w}{\left(\calQ^{\half}_{i,\star}S_i + S^T_i\calQ^{\half}_{i,\star}\right) w} \geq 0,
    \label{eq: ellipsoidal approx: under-approx BRS proof 8}
\end{equation}
which holds by Lemma \ref{lemma: ellipsoidal approx: lower bound on norm over transformed unit ball}. Thus, for all $i\in\{1,\ldots,n_q\}$,  $\overline{e}_i$ satisfies the PDI in \eqref{eq: approx reach: HJB PDE under approx BRS collection}. By Lemma \ref{lemma: approx reach: min and max function over supersolutions and subsolutions},
\begin{equation*}
    \overline{e}(t,x) \doteq \min_{1 \leq i \leq n_q}\overline{e}_i(t,x),\quad \forall (t,x)\in[0,T]\times\R^n,
\end{equation*}
is a viscosity supersolution of the HJB PDE in \eqref{eq: approx reach: HJB PDE} with $\Omega = \overline{e}^{-1}((-\infty,0))$ and the terminal condition $\overline{e}(T,x) = \iprod{x-x_e}{X^{-1}_e(x-x_e)} -1$ for all $x\in\R^n$. Finally, since Lemma \ref{lemma: ellipsoidal approx: min over quadratics attached sublevel set} ensures that $x\mapsto \overline{e}(t,x)$ is of class $\cZ$ for any $t\in[0,T]$, Theorem \ref{theorem: approx reach: HJB equation for under-approximate BRS} then implies that $\{x\in\R^n \,|\, \overline{e}(t,x) \leq 0 \} = \bigcup^{n_q}_{i=1}\mathcal{E}\left(q_i(t), Q_i(t)\right)\subseteq \mathcal{G}(t)$ for all $t\in[0,T]$. 

\underline{Assertion \emph{(\ref{item: ellipsoidal approx: under-approx BRS assertion 2})}}. Fix any $t\in[0,T]$ and any $i\in\{1,\ldots, n_q\}$ for which $S_i\in\cC\left([0,T];\,\mathbb{O}^n\right)$ satisfies \eqref{eq: ellipsoidal approx: under approx BRS orthogonal matrix} with $x^\star(\cdot)$ being a solution of \eqref{eq: ellipsoidal approx: linear system}. Denote by $u^\star \in \cU[t,T]$ the control corresponding to $x^\star(\cdot)$. Since $\overline{e}_i\in\cC^1([0,T]\times\R^n;\,\R)$, differentiation of $\overline{e}_i$ along $x^\star(\cdot)$ yields a.e. $s\in[t,T]$
\begin{align*}
    -\tfrac{d}{ds}\overline{e}_i(s, x^\star(s)) &= -\tfrac{\partial}{\partial s}\overline{e}_i(s,x^\star(s)) + \iprod{-\grad \overline{e}_i(s,x^\star(s))}{\tfrac{d}{ds}x^\star(s)} \\
    &\leq -\tfrac{\partial}{\partial s}\overline{e}_i(s,x^\star(s)) +\max_{u\in\U}\iprod{-\grad \overline{e}_i(s,x^\star(s))}{A(s)x^\star(s) + B(s)u}.
\end{align*}
Hence, by \eqref{eq: ellipsoidal approx: under-approx BRS proof 4}-\eqref{eq: ellipsoidal approx: under-approx BRS proof 7},
\begin{align}
    &-\tfrac{d}{ds}\overline{e}_i(s, x^\star(s))  \leq 2\bnorml{\calQ^{\half}_{i,\star}\hat{w}_i}
    - \iprod{\hat{w}_i}{\left(\calQ^{\half}_{i,\star}S_i + S^T_i\calQ^{\half}_{i,\star}\right)\hat{w}_i},
        \label{eq: ellipsoidal approx: under-approx BRS proof 9}
\end{align}
with $\hat{w}_i$ as defined in \eqref{eq: ellipsoidal approx: hat w for orthogonal matrix}. Consider the cases: \emph{i)} $\bnorml{\calQ^{\half}_{i,\star}\hat{w}_i}\neq 0$ and \emph{ii)} $\bnorml{\calQ^{\half}_{i,\star}\hat{w}_i}= 0$. Beginning with case \emph{i)}, first note that $\bnorml{\calQ^{\half}_{i,\star}\hat{w}_i} \neq 0$ implies $\norml{\hat{w}_i}\neq 0$. From \eqref{eq: ellipsoidal approx: under approx BRS orthogonal matrix},
\begin{align*}
     & \frac{1}{\norml{\hat{w}_i}}\iprod{\calQ^{\half}_{i,\star}\hat{w}_i}{S_i \hat{w}_i}  = \frac{1}{\bnorml{\calQ^{\half}_{i,\star}\hat{w}_i}} \iprod{\calQ^{\half}_{i,\star}\hat{w}_i}{ \calQ^{\half}_{i,\star}\hat{w}_i}\\
     &\implies \iprod{\hat{w}_i}{\left(\calQ^{\half}_{i,\star} S_i + S^T_i\calQ^{\half}_{i,\star}\right)\hat{w}_i} = 2\norml{\hat{w}_i}\bnorml{\calQ^{\half}_{i,\star}\hat{w}_i}.
\end{align*}
Thus, for $\overline{e}_i(s,x^\star(s)) \geq 0$, or equivalently $\norml{\hat{w}_i} = \norml{Q^{-\half}_i(x^\star - q_i)} \geq 1$ (see Remark \ref{remark: ellipsoidal approx: ellipsoid is affine transformation}), the following holds
\begin{equation}
2\norml{\calQ^{\half}_{i,\star}\hat{w}_i} -  \iprod{\hat{w}_i}{\left(\calQ^{\half}_{i,\star} S_i + S^T_i\calQ^{\half}_{i,\star}\right)\hat{w}_i} \leq 0.
\label{eq: ellipsoidal approx: under-approx BRS proof 10}
\end{equation}

Now consider case \emph{ii)} where $\bnorml{\calQ^{\half}_{i,\star}\hat{w}_i } = 0$, which implies $\calQ^{\half}_{i,\star}\hat{w}_i = 0$ and by substitution into the right-hand side of \eqref{eq: ellipsoidal approx: under-approx BRS proof 9} the inequality in \eqref{eq: ellipsoidal approx: under-approx BRS proof 10} also holds. Then, from \eqref{eq: ellipsoidal approx: under-approx BRS proof 9},
\begin{equation}
    -\tfrac{d}{ds}\overline{e}_i(s, x^\star(s)) \leq 0, \quad \forall s \in \{s \in [t, T]\,|\, \overline{e}_i(s, x^\star(s)) \geq 0\},
    \label{eq: ellipsoidal approx: under-approx BRS proof 11}
\end{equation}
which implies that $\overline{e}_i(s,x^\star(s)) \leq 0$ for all $s\in [t,T]$. To demonstrate this, assume that the contrary is true, that is, $\overline{e}_i(s_1,x^\star(s_1)) > 0$ for some $s_1 \in [t,T]$. Since $x^\star(T) \in\mathcal{X}$ implies $\overline{e}_i(T,x^\star(T)) = g(x^\star(T)) \leq 0$, there exists a time $s_2 \in (s_1, T]$ such that $\overline{e}_i(s_2, x^\star(s_2))$ is non-positive. Let $s_2$ be the first such time, that is, $s_2 \doteq \min\{s \in (s_1,T]\,|\,\overline{e}_i(s, x^\star(s))\leq 0\}$. Then, $\overline{e}_i(s, x^\star(s)) \geq 0$ for all $s \in [s_1, s_2]$. From \eqref{eq: ellipsoidal approx: under-approx BRS proof 11},
\begin{align*}
    0&= \overline{e}_i(s_2, x^\star(s_2)) =\overline{e}_i(s_1, x^\star(s_1)) + \integral{s_1}{s_2}{\tfrac{d}{ds}\overline{e}_i(s, x^\star(s))}{s}\geq \overline{e}_i(s_1, x^\star(s_1)) ,
\end{align*}
which yields the desired contradiction. Thus, $\overline{e}_i(s,x^\star(s)) \leq 0$ for all $s\in[t,T]$, which implies $x^\star(t) \in \mathcal{E}(q_i(t), Q_i(t))$. 
\end{proof}

\begin{remark} Assumption \ref{assumption: ellipsoidal approx: S matrix choice for positive definiteness} is asserted to ensure that the ellipsoids $\mathcal{E}(q_i(t), Q_i(t))$ in \eqref{eq: ellipsoidal approx: under-approx BRS as a union of ellipsoids} are non-degenerate for $t\in[0,T]$. It holds in the special case where $S_i(t) = \I$ for all $t\in[0,T]$, whereupon Theorem \ref{theorem: ellipsoidal approx: under-approx BRS} recovers \cite[Theorem 11]{LMD23}. It should be noted, however, that $S_i\in\cC([0,T];\,\mathbb{O}^n)$ is not a sufficient condition for Assumption \ref{assumption: ellipsoidal approx: S matrix choice for positive definiteness} to hold. For example, consider the single integrator system
\begin{equation*}
    \Dot{x}(s) = u(s), \quad \text{a.e. } s \in (0,1), 
\end{equation*}
with $\U = \mathcal{E}(0, 1)$ and $\mathcal{X} =\mathcal{E}(0, 0.1^2)$. Then, taking $T = 1$ and selecting $S_i(t) \equiv -1$ for all $t\in[0,1]$, the FVP for $Q_i$ reduces to 
\begin{equation}
    \Dot{Q}_i(t) = 2Q^{\half}_i(t), \quad Q_i(1) = 0.01.
    \label{eq: ellipsoidal approx: remark regarding existence of positive definite Q}
\end{equation}
It can be verified that a solution of \eqref{eq: ellipsoidal approx: remark regarding existence of positive definite Q} is
\begin{equation*}
    Q_i(t) = \begin{cases} (t-0.9)^2, \quad  t\in[0.9, 1], \\0, \quad t\in[0,0.9),
    \end{cases}
\end{equation*}
however, $Q_i(t)$ is not positive definite on the interval $[0,1]$. A general analysis of existence and uniqueness of solutions for \eqref{eq: ellipsoidal approx: under approx BRS dynamics + terminal condition}-\eqref{eq: ellipsoidal approx: under approx BRS Q star matrix} is not considered in the current work.
\label{remark: ellipsoidal approx: positive definiteness of ellipsoid shape}
\end{remark}

The orthogonal matrix $S_i(t)$ in \eqref{eq: ellipsoidal approx: under approx BRS Qu matrix} is a generalisation that allows for any solution of \eqref{eq: ellipsoidal approx: linear system} to be contained in each under-approximating ellipsoid $\mathcal{E}(q_i(t),Q_i(t))$. Intuitively, if a solution that lies on the boundary of the exact backwards reachable set $\mathcal{G}(t)$ is selected, then the boundaries of $\mathcal{E}(q_i(t),Q_i(t))$ and $\mathcal{G}(t)$ must coincide. For the problem setting considered in this work, Pontryagin's maximum principle provides both necessary and sufficient conditions for optimality and can hence be used to generate solutions that lie on $\partial\mathcal{G}(t)$. In particular, boundary solutions $x^\star(\cdot)$ of \eqref{eq: ellipsoidal approx: linear system} can be obtained via the following final value problem
\begin{equation}
\left\{\;
\begin{split}
    &\Dot{x}^\star(s) = A(s)x^\star(s) + B(s)u^\star(s),\\
    &\Dot{\lambda}(s) = -A^T(s)\lambda(s),\\
    &u^\star(s) \in \argmax_{u\in\U}\iprod{-\lambda(s)}{B(s)u}, \\
    &x^\star(T) \in\partial \mathcal{X}, \quad \lambda(T) = 2X^{-1}_e(x^\star(T)-x_e).
    \end{split}\right. 
    \label{eq: ellipsoidal approx: PMP backwards dynamics for LTV system}
\end{equation}
This would then recover the tightness property discussed in \cite{kurzhanski2002reachability} without using arguments involving set operations.
\begin{corollary}
       Consider the statements in Theorem \ref{theorem: ellipsoidal approx: under-approx BRS} and let all of its assumptions hold. If, in addition to the statements of assertion (\ref{item: ellipsoidal approx: under-approx BRS assertion 2}), $x^\star(\cdot)$ is a solution of \eqref{eq: ellipsoidal approx: linear system} satisfying the FVP in \eqref{eq: ellipsoidal approx: PMP backwards dynamics for LTV system} a.e. $s\in(t,T)$ with $u^\star \in \cU[t,T]$, then, $x^\star(t) \in \partial\mathcal{E}(q_i(t), Q_i(t)) \cap \partial\mathcal{G}(t)$ for all $t\in[0,T]$.
\label{corollary: ellipsoidal approx: tight under-approx of BRS}
\end{corollary}
\begin{proof}
Fix any $t\in[0,T]$, $\hat{u}\in\cU[t,T]$, and $x^\star(T) \in \partial\mathcal{X}$. Let $\hat{x}(\cdot) \doteq \varphi(\cdot;t,x^\star(t),\hat{u})$ be the corresponding solution of \eqref{eq: ellipsoidal approx: linear system} evolved from $\hat{x}(t) = x^\star(t)$ under $\hat{u}$. Since $x\mapsto g(x) = \iprod{x-x_e}{X^{-1}_e(x-x_e)}$ is continuously differentiable and convex, 
\begin{equation}
    g(\hat{x}(T)) \geq g(x^\star(T)) + \iprod{\nabla g(x^\star(T))}{\hat{x}(T)-x^\star(T)}.
   \label{eq: ellipsoidal approx: boundary of BRS via PMP proof 1}
\end{equation}
Let $\Phi_A:[t,T]\times[t,T]\rightarrow\R^{n\times n}$ denote the state-transition matrix associated with the dynamics $\dot{x}(s) = A(s)x(s)$, $s\in (t,T)$. Noting that $\left(\Phi_A(T,s)\right)^T = \Phi_{-A^T}(s,T)$ (see for example \cite[Theorem A.4.5, p.291]{CS:04}) and
\begin{align*}
    \hat{x}(T) - x^\star(T) &= \Phi_{A}(T, t)x^\star(t) + \integral{t}{T}{\Phi_{A}(T, s)B(s)\hat{u}(s)}{s} \\
     &-\Phi_{A}(T, t)x^\star(t) - \integral{t}{T}{\Phi_{A}(T, s)B(s)u^\star(s)}{s},
\end{align*}
the right-most term in \eqref{eq: ellipsoidal approx: boundary of BRS via PMP proof 1} can be rewritten as 
\begin{align}
    \iprod{\nabla g(x^\star(T))}{\hat{x}(T)-x^\star(T)} &= \iprod{\nabla g(x^\star(T))}{\integral{t}{T}{\Phi_{A}(T, s)B(s)(\hat{u}(s)-u^\star(s))}{s}} \nn\\
    &= \integral{t}{T}{\iprod{\Phi_{-A^T}(s,T)\grad g(x^\star(T))}{B(s)(\hat{u}(s)-u^\star(s))}}{s}.\label{eq: ellipsoidal approx: boundary of BRS via PMP proof 2}
\end{align}
Furthermore, $\lambda(s) = \Phi_{-A^T}(s,T)\grad g(x^\star(T))$, $s\in[t,T]$, with $\grad g(x^\star(T)) = 2X^{-1}_e(x^\star(T) - x_e)$ is the unique adjoint solution to the FVP in \eqref{eq: ellipsoidal approx: PMP backwards dynamics for LTV system}. As $u^\star(s)$ is, almost everywhere, a maximiser of $\iprod{-\lambda(s)}{B(s)u^\star}$ on $\U$, $\iprod{-\lambda(s)}{B(s)u^\star} \geq \iprod{-\lambda(s)}{B(s)\hat{u}}$ a.e. $s\in(t,T)$. Continuing from \eqref{eq: ellipsoidal approx: boundary of BRS via PMP proof 2},
\begin{equation}
     \iprod{\nabla g(x^\star(T))}{\hat{x}(T)-x^\star(T)} = \integral{t}{T}{\iprod{-\lambda(s)}{B(s)(u^\star(s)-\hat{u}(s))}}{s} \geq 0.
     \label{eq: ellipsoidal approx: boundary of BRS via PMP proof 3}
\end{equation}
Substitution of \eqref{eq: ellipsoidal approx: boundary of BRS via PMP proof 3} in \eqref{eq: ellipsoidal approx: boundary of BRS via PMP proof 1} yields
\begin{equation}
    g(x^\star(T)) \leq g(\hat{x}(T)) = g(\varphi(T;t,x^\star(t), \hat{u})).
     \label{eq: ellipsoidal approx: boundary of BRS via PMP proof 4}
\end{equation}
As $\hat{u}\in\cU[t,T]$ is arbitrary, taking the inf over \eqref{eq: ellipsoidal approx: boundary of BRS via PMP proof 4} yields
\begin{equation}
   g(x^\star(T)) \leq \inf_{u\in\cU[t,T]}g(\varphi(T;t,x^\star(t),u)) = v(t,x^\star(t)),\label{eq: ellipsoidal approx: boundary of BRS via PMP proof 5} 
\end{equation}
where $v$ is the value function \eqref{eq: background: value function} for the LTV system \eqref{eq: ellipsoidal approx: linear system}. By optimality, $v(t,x^\star(t)) \leq g(x^\star(T))$ and together with \eqref{eq: ellipsoidal approx: boundary of BRS via PMP proof 5} yields $g(x^\star(T))=v(t,x^\star(t))$. Moreover, since 
\begin{equation*}
    \partial \mathcal{X} = \{x\in\R^n\,|\,\iprod{x-x_e}{X^{-1}_e(x-x_e)} = 1\}
\end{equation*}
and $x^\star(T) \in \partial \mathcal{X}$, then $g(x^\star(T))=0$, which also implies $v(t,x^\star(t)) = 0$. From Theorem \ref{theorem: background: HJB equation for BRS}, it follows that $x^\star(t)\in \mathcal{G}(t)$. Moreover, $g$ is of class $\cZ$ (see Lemma \ref{lemma: ellipsoidal approx: min over quadratics attached sublevel set} with $n_q = 1$), so by Lemma \ref{lemma: approx reach: value function satisfies attached sublevel set}, $x\mapsto v(t,x)$ must also be of class $\cZ$. Then, by Lemma \ref{lemma: approx reach: closure and interior of sublevel sets},
\begin{align}
    &x^\star(t) \notin  \{x\in\R^n\,|\,v(t,x) < 0\} = \text{int}\left(\mathcal{G}(t)\right), \nn \\
    &\implies x^\star(t) \in \partial\mathcal{G}(t). \label{eq: ellipsoidal approx: boundary of BRS via PMP proof 6}
\end{align} 

To show that $x^\star(t) \in \partial \mathcal{E}(q_i(t),Q_i(t))$, first note that by Theorem \ref{theorem: ellipsoidal approx: under-approx BRS} assertion \emph{(b)},  $x^\star(t) \in \mathcal{E}(q_i(t),Q_i(t))$, which implies $x^\star(t)\notin \text{int}\left(\mathcal{E}(q_i(t),Q_i(t))\right)$. To see this, assume that the contrary is true, i.e., $x^\star(t)\in \text{int}\left(\mathcal{E}(q_i(t),Q_i(t))\right)$. From Theorem \ref{theorem: ellipsoidal approx: under-approx BRS} assertion \emph{(a)}, $\mathcal{E}(q_i(t),Q_i(t))\subseteq \mathcal{G}(t)$, which implies that $\text{int}\left(\mathcal{E}(q_i(t),Q_i(t))\right)\subseteq \text{int}\left(\mathcal{G}(t)\right)$ as the interior operation preserves subset relationships. So $x^\star(t)\in\text{int}\left(\mathcal{G}(t)\right)$, which contradicts \eqref{eq: ellipsoidal approx: boundary of BRS via PMP proof 6}. Thus, $x^\star(t) \notin  \text{int}\left(\mathcal{E}(q_i(t),Q_i(t))\right)$ so $ x^\star(t) \in \partial \mathcal{E}(q_i(t),Q_i(t))$ for any $t\in[0,T]$. 
\end{proof}
\begin{remark}
   We propose one method to ensure that the orthogonal matrix $S_i(s)$ can be selected to satisfy \eqref{eq: ellipsoidal approx: under approx BRS orthogonal matrix} in Theorem \ref{theorem: ellipsoidal approx: under-approx BRS}. With time dependencies omitted, consider the case where $ \calQ^{\half}_{i,\star}\hat{w}_i\neq 0$, which implies $\hat{w}_i\neq 0$. Let 
   \begin{equation*}
       w \doteq \frac{\hat{w}_i}{\bnorml{\hat{w}_i}}, \quad w^\star \doteq \frac{\calQ^{\half}_{i,\star}\hat{w}_i }{\bnorml{\calQ^{\half}_{i,\star}\hat{w}_i}}.
   \end{equation*}
Since $w$ and $w^\star$ are both of unit norm, the problem of finding $S_i(s)$ to satisfy \eqref{eq: ellipsoidal approx: under approx BRS orthogonal matrix} is equivalent to finding an orthogonal matrix $S$ that will rotate $w$ to $w^\star$. Consider the case where $w \neq w^\star$ so that an orthonormal basis for the span of $\{w, w^\star\}$ can be constructed via the Gram-Schmidt process. Let 
\begin{equation*}
    r \doteq w, \quad r_\perp \doteq \frac{w^\star - \iprod{w^\star}{w}w}{\norml{w^\star - \iprod{w^\star}{w}w}}, 
\end{equation*}
and let $\cos{\theta} = \iprod{r}{w^\star}$ and $\sin{\theta} = \iprod{r_\perp}{w^\star}$ where $\theta$ is the angle of rotation between $w$ and $w^\star$. Consider the Jacobi rotation matrix \cite{constantine1978some}
    \begin{equation}
        S = \I + \sin{\theta}\left(r_\perp r^T - rr^T_{\perp}\right) + (\cos{\theta}-1)\left(rr^T + r_{\perp}r^T_{\perp}\right).
        \label{eq: remark regarding choice of S matrix}
    \end{equation}
    It can be readily verified that $S\in \mathbb{O}^n$ and $Sw = w^\star$, which yields the desired selection of $S_i(s)$ in \eqref{eq: ellipsoidal approx: under approx BRS orthogonal matrix}. In the case where $w = w^\star$, the selection $S = \I$ ensures that \eqref{eq: ellipsoidal approx: under approx BRS orthogonal matrix} holds, which corresponds to a zero rotation angle.
    \label{remark: ellipsoidal approx: rodrigues rotation formula}
    \end{remark}
\subsection{Ellipsoidal Over-approximation of Backwards Reachable Sets}
We now move on to producing an analogous result to Theorem \ref{theorem: ellipsoidal approx: under-approx BRS} for over-approximating the backwards reachable set; this time by the intersection of ellipsoidal sets. An over-approximation of $\mathcal{G}(t)$ can be produced by evolving each ellipsoid $\mathcal{E}(q_i(t),Q_i(t))$ in \eqref{eq: ellipsoidal approx: max over quadratics is intersect} backwards in time according to the FVP defined via \eqref{eq: ellipsoidal approx: linear system} by
\begin{equation}
\left\{\;
\begin{split}
    \dot{q}_i(t) &= A(t)q_i(t) + B(t)p,\\
    \dot{Q}_i(t) &= A(t)Q_i(t)+Q_i(t)A^T(t) -\overline{\calR}_i(t,Q_i(t)),  \\
    q_i(T) &= x_e, \quad Q_i(T) = X_e,
    \end{split} \right.
    \label{eq: ellipsoidal approx: over approx BRS dynamics + terminal condition}
\end{equation}
for all $i\in\{1,\ldots,n_q\}$, a.e. $t\in(0,T)$, in which the operator $\overline{\calR}_i:[0,T]\times\spd{n}\rightarrow\spd{n}$ is defined by
\begin{equation}
    \overline{\calR}_i(t,Q) \doteq \frac{1}{\kappa_i(t)}B(t)PB^T(t) + \kappa_i(t)Q(t),
    \label{eq: ellipsoidal approx: over approx BRS forcing matrix}
\end{equation}
with $\kappa_i:[0,T]\rightarrow \R_{>0}$ being continuous. 
\begin{theorem} Consider the LTV system described by \eqref{eq: ellipsoidal approx: linear system}.
\begin{enumerate}[(a)]
    \item The FVP in \eqref{eq: ellipsoidal approx: over approx BRS dynamics + terminal condition}-\eqref{eq: ellipsoidal approx: over approx BRS forcing matrix} admits a unique solution with $q_i(t)\in\R^n$ and $Q_i(t)\in\spd{n}$ for all $t\in[0,T]$ and for all $i\in\{1,\ldots, n_q\}$. \label{item: ellipsoidal approx: over-approx BRS assertion 1}
    \item The set
\begin{equation}
    \overline{\mathcal{G}}(t) \doteq \bigcap^{n_q}_{i=1}\mathcal{E}\left(q_i(t), Q_i(t)\right)
    \label{eq: ellipsoidal approx: over-approx BRS as an intersect of ellipsoids}
\end{equation}
is an over-approximation of $\mathcal{G}(t)$ for \eqref{eq: ellipsoidal approx: linear system} for all $t\in[0,T]$.\label{item: ellipsoidal approx: over-approx BRS assertion 2}
\end{enumerate}
\label{theorem: ellipsoidal approx: over-approx BRS}
\end{theorem}
\begin{proof}
\underline{Assertion \emph{(\ref{item: ellipsoidal approx: over-approx BRS assertion 1})}}. Continuity of $A(\cdot)$ and $B(\cdot)$ implies that the map $q\mapsto A(t)q + B(t)p$ is Lipschitz continuous in $q$, uniformly in $t$. Existence and uniqueness of solutions for $q_i$ then follows, see for example \cite[Theorem 3.2, p.93]{khalil2002nonlinear}. Omitting time dependencies for brevity, 
\begin{equation}
    \dot{Q}_i = \left(A - \tfrac{1}{2}\kappa_i\I\right)Q_i+Q_i\left(A - \tfrac{1}{2}\kappa_i\I\right)^T-\tfrac{1}{\kappa_i}BPB^T.
    \label{eq: ellipsoidal approx: over-approx BRS proof 0}
\end{equation}
A unique solution of \eqref{eq: ellipsoidal approx: over-approx BRS proof 0} with the terminal condition $Q_i(T) = X_e$ exists, see for example \cite[Theorem 1]{behr2019solution}, and is given by
\begin{align}
    Q_i(t) &= \Phi_{i}(t,T)X_e \Phi^T_{i}(t,T) - \integral{T}{t}{\tfrac{1}{\kappa_i(s)}\Phi_{i}(t, s)B(s)PB^T(s)\Phi^T_{i}(t,s)}{s},
    \label{eq: ellipsoidal approx: over-approx BRS proof 1}
\end{align}
for all $t\in[0,T]$, where $\Phi_{i}:[t,T]\times[t,T]\rightarrow\R^{n\times n}$ denotes the state-transition matrix associated with $\dot{x}(s) = \left(A(s) - \frac{1}{2}\kappa_i(s)\I\right)x(s)$, $s\in(t,T)$. Since $X_e\in\spd{n}$ implies $\Phi_i(t,T)X_e \Phi^T_i(t,T)\in\spd{n}$, and $P\in\spd{m}$ implies $\Phi_{i}(t, s)B(s)PB^T(s)\Phi^T_{i}(t,s)\in\spsd{n}$, it follows from \eqref{eq: ellipsoidal approx: over-approx BRS proof 1} that $Q_i(t)\in\spd{n}$ for all $t\in[0,T]$.

\underline{Assertion \emph{(\ref{item: ellipsoidal approx: over-approx BRS assertion 2})}}. Fix any index $i\in\{1,\ldots,n_q\}$ and let $\underline{e_i}\in\cC^1([0,T]\times\R^n;\,\R)$ be defined by
\begin{equation}
    \underline{e_i}(t,x) \doteq \iprod{x-q_i(t)}{Q^{-1}_i(t)(x-q_i(t))} -1,
    \label{eq: ellipsoidal approx: over-approx BRS proof 2}
\end{equation}
where $\{q_i, Q_i\}$ is the solution of \eqref{eq: ellipsoidal approx: over approx BRS dynamics + terminal condition}-\eqref{eq: ellipsoidal approx: over approx BRS forcing matrix}. Following similarly to Theorem \ref{theorem: ellipsoidal approx: under-approx BRS}, substitution of \eqref{eq: ellipsoidal approx: over approx BRS dynamics + terminal condition}-\eqref{eq: ellipsoidal approx: over approx BRS forcing matrix} in \eqref{eq: ellipsoidal approx: under-approx BRS proof 6} yields
\begin{align}
        -\tfrac{\partial}{\partial t}\underline{e_i} + H(t, x, \nabla\underline{e_i}(t,x))  &= 2\sqrt{\iprod{x-q_i}{Q^{-1}_iBPB^TQ^{-1}_i(x-q_i)}} \nn\\
        & - \iprod{x - q_i}{Q^{-1}_i\left(\tfrac{1}{\kappa_i}BPB^T + \kappa_iQ_i\right)Q^{-1}_i(x - q_i)}.
    \label{eq: ellipsoidal approx: over-approx BRS proof 3}
\end{align}
Substitution of $x = Q^\half_i(t)w + q_i(t)$ in \eqref{eq: ellipsoidal approx: over-approx BRS proof 3} yields
\begin{align*}
    &2\sqrt{\iprod{w}{Q^{-\half}_iBPB^TQ^{-\half}_i w}}- \iprod{w}{\left(\tfrac{1}{\kappa_i}Q^{-\half}_iBPB^TQ^{-\half}_i + \kappa_i\I\right) w} \leq 0,
\end{align*}
in which the inequality holds over $\norml{w} > 1$ by Lemma \ref{lemma: ellipsoidal approx: upper bound on norm over transformed unit ball}. Using Remark \ref{remark: ellipsoidal approx: ellipsoid is affine transformation}, $\underline{e_i}$ satisfies the PDI in \eqref{eq: approx reach: HJB PDE over approx BRS collection} for all $i\in\{1,\ldots,n_q\}$. Then, by Lemma \ref{lemma: approx reach: min and max function over supersolutions and subsolutions}, 
\begin{equation*}
    \underline{e}(t,x) \doteq \max_{1 \leq i \leq n_q}\underline{e_i}(t,x),\quad \forall (t,x)\in[0,T]\times\R^n,
\end{equation*}
is a viscosity subsolution of \eqref{eq: approx reach: HJB PDE} with $\Omega = \underline{e}^{-1}((0,\infty))$ and the terminal condition 
$\underline{e}(T,x) = \iprod{x-x_e}{X^{-1}_e(x-x_e)} - 1,$ for all $x\in\R^n$.
Finally, using Theorem \ref{theorem: approx reach: HJB equation for over-approximate BRS}, the zero sublevel set $\{x\in\R^n \,|\, \underline{e}(t,x) \leq 0 \} = \bigcap^{n_q}_{i=1}\mathcal{E}\left(q_i(t), Q_i(t)\right)$ is an over-approximation of $\mathcal{G}(t)$ for all $t\in[0,T]$. 
\end{proof}

Likewise with $S_i$ in \eqref{eq: ellipsoidal approx: under approx BRS dynamics + terminal condition}, the scalar function $\kappa_i$ allows for a family of over-approximating ellipsoids $\mathcal{E}(q_i(t),Q_i(t))$ to be produced. When chosen appropriately, this generalisation allows the boundaries of the over-approximating ellipsoids and $\mathcal{G}(t)$ to coincide along a solution of \eqref{eq: ellipsoidal approx: linear system}, which is analogous to Corollary \ref{corollary: ellipsoidal approx: tight under-approx of BRS}. In particular, the desired selection of $\kappa_i$ is
\begin{equation}
    \kappa_i(s) = \frac{\bnorml{\calQ^{\half}_{i,\star}(s)\hat{w}_i(s)}}{\bnorml{\hat{w}_i(s)}}, 
    \label{eq: ellipsoidal approx: over approx BRS kappa choice}
\end{equation}
for all $s\in[t,T]$ such that $\hat{w}_i(s) \neq 0$, in which $\calQ^{\half}_{i,\star}$ and $\hat{w}_i$ are defined in \eqref{eq: ellipsoidal approx: under approx BRS Q star matrix} and \eqref{eq: ellipsoidal approx: hat w for orthogonal matrix}, respectively, with $x^\star(\cdot)$ being a solution of \eqref{eq: ellipsoidal approx: linear system} satisfying the FVP 
\begin{equation}
\left\{
    \begin{split}
        \dot{x}^\star(s) &= A(s)x^\star(s) + B(s)u^\star(s), \quad x^\star(T) \in \partial \mathcal{X},\\
        u^\star(s) &\in \argmax_{u\in\U}\iprod{-Q^{-\half}_i(s)\hat{w}_i(s)}{B(s)u}.
    \end{split}\right.
    \label{eq: ellipsoidal approx: over approx boundary soln dynamics}
\end{equation}
\begin{corollary} Consider the statements in Theorem \ref{theorem: ellipsoidal approx: over-approx BRS}. Given $\kappa_i\in\cC\left([0,T];\R_{>0}\right)$ with $i\in\{1,\ldots,n_q\}$ satisfying \eqref{eq: ellipsoidal approx: over approx BRS kappa choice} for all $s\in[t,T]$ such that $\hat{w}_i(s) \neq 0$ and $x^\star(\cdot)$ is a solution of \eqref{eq: ellipsoidal approx: linear system} satisfying \eqref{eq: ellipsoidal approx: over approx boundary soln dynamics} a.e. $s\in(t,T)$ with $u^\star \in \cU[t,T]$, then, $x^\star(t) \in \partial \mathcal{E}(q_i(t),Q_i(t)) \cap \partial\mathcal{G}(t)$ for all $t\in[0,T]$.
\label{corollary: ellipsoidal approx: tight over-approx of BRS}
\end{corollary}
\begin{proof}
Fix any $t\in[0,T]$ and let $i\in\{1,\ldots,n_q\}$ be any index for which $\kappa_i\in\cC\left([0,T];\,\R_{>0}\right)$ satisfies \eqref{eq: ellipsoidal approx: over approx BRS kappa choice} with $x^\star(\cdot)$ satisfying \eqref{eq: ellipsoidal approx: over approx boundary soln dynamics}. Consider $\underline{e_i}\in\cC^1\left([t,T]\times\R^n;\,\R\right)$ as defined in \eqref{eq: ellipsoidal approx: over-approx BRS proof 2} over the restricted domain $[t,T]\times\R^n$. Differentiation of $\underline{e_i}$ along $x^\star(\cdot)$ yields a.e. $s \in [t,T]$,
 \begin{align}
      &-\tfrac{d}{ds}\underline{e_i}(s, x^\star(s)) = -\tfrac{\partial}{\partial s}\underline{e_i}(s, x^\star(s))  +\iprod{-\grad \underline{e_i}(s,x^\star(s))}{A(s)x^\star(s)+B(s)u^\star(s)}.
      \label{eq: ellipsoidal approx: tight over-approx of BRS proof 2}
 \end{align}
 Since $\grad \underline{e_i}(s, x^\star(s)) = 2Q^{-1}_i(s)(x^\star(s) - q_i(s))$, from \eqref{eq: ellipsoidal approx: over approx boundary soln dynamics} and \eqref{eq: ellipsoidal approx: hat w for orthogonal matrix}, $u^\star(s)$ is a maximiser of the right-most term in \eqref{eq: ellipsoidal approx: tight over-approx of BRS proof 2} with respect to $u$. Continuing from \eqref{eq: ellipsoidal approx: tight over-approx of BRS proof 2}
 \begin{align}
      &-\tfrac{d}{ds}\underline{e_i}(s, x^\star(s)) = -\tfrac{\partial}{\partial s}\underline{e_i}(s, x^\star(s)) + \max_{u\in\U}\iprod{-\grad \underline{e_i}(s,x^\star(s))}{A(s)x^\star(s)+B(s)u}.
      \label{eq: ellipsoidal approx: tight over-approx of BRS proof 3}
 \end{align}
From \eqref{eq: ellipsoidal approx: over-approx BRS proof 3} and \eqref{eq: ellipsoidal approx: hat w for orthogonal matrix},
\begin{align}
    -&\tfrac{d}{ds}\underline{e_i}(s, x^\star(s))  =2\bnorml{\calQ^{\half}_{i,\star}\hat{w}_i} 
    - \iprod{\hat{w}_i}{Q\left(\tfrac{1}{\kappa_i}\calQ_{i,\star}+\kappa_i\I\right)\hat{w}_i}.
    \label{eq: ellipsoidal approx: tight over-approx of BRS proof 4}
\end{align}

The right-hand side of \eqref{eq: ellipsoidal approx: tight over-approx of BRS proof 4} is non-negative whenever $\bnorml{\hat{w}_i} \leq 1$. To see this, consider the cases: \emph{i)} $0<\norml{\hat{w}_i}\leq 1$ and \emph{ii)} $\norml{\hat{w}_i}=0$. Beginning with case \emph{i)}, 
\begin{align*}
    &2\bnorml{\calQ^{\half}_{i,\star}\hat{w}_i} - \iprod{\hat{w}_i}{\left(\tfrac{1}{\kappa_i}\calQ_{i,\star}+\kappa_i\I\right)\hat{w}_i}  = \norml{\hat{w}_i}\Big(2\frac{\bnorml{\calQ^{\half}_{i,\star}\hat{w}_i}}{\norml{\hat{w}_i}}- \iprod{\frac{\hat{w}_i}{\norml{\hat{w}_i}}}{\left(\tfrac{1}{\kappa_i}\calQ_{i,\star}+\kappa_i\I\right) \hat{w}_i}\Big).
\end{align*} 
As $(\tfrac{1}{\kappa_i}\calQ_{i,\star}+\kappa_i\I)\succ 0$ and noting that $\bnorml{\hat{w}_i} \leq 1$,
\begin{align*}
    &2\bnorml{\calQ^{\half}_{i,\star}\hat{w}_i} - \iprod{\hat{w}_i}{\left(\tfrac{1}{\kappa_i}\calQ_{i,\star}+\kappa_i\I\right)\hat{w}_i} \geq \norml{\hat{w}_i}\Big(2\frac{\bnorml{\calQ^{\half}_{i,\star}\hat{w}_i}}{\norml{\hat{w}_i}}- \iprod{\frac{\hat{w}_i}{\norml{\hat{w}_i}}}{\left(\tfrac{1}{\kappa_i}\calQ_{i,\star}+\kappa_i\I\right) \frac{\hat{w}_i}{\norml{\hat{w}_i}}}\Big).
\end{align*} 
Next, using $\hat{w}_i\slash \norml{\hat{w}_i}$ in place of $w^\star$ in Lemma \ref{lemma: ellipsoidal approx: upper bound on norm over transformed unit ball}, the choice of $\kappa_i$ in \eqref{eq: ellipsoidal approx: over approx BRS kappa choice} results in
\begin{align}
    &2\frac{\bnorml{\calQ^{\half}_{i,\star}\hat{w}_i}}{\norml{\hat{w}_i}}- \iprod{\frac{\hat{w}_i}{\norml{\hat{w}_i}}}{\left(\tfrac{1}{\kappa_i}\calQ_{i,\star}+\kappa_i\I\right) \frac{\hat{w}_i}{\norml{\hat{w}_i}}} = 0 \nn \\
    &\implies  2\bnorml{Q^{\half}_{i,\star}\hat{w}_i} - \iprod{\hat{w}_i}{\left(\tfrac{1}{\kappa_i}Q_{i,\star}+\kappa_i\I\right)\hat{w}_i} \geq 0. \label{eq: ellipsoidal approx: tight over-approx of BRS proof 5}
\end{align}

Now consider case \emph{ii)} where $\norml{\hat{w}_i} = 0$, which implies that $\hat{w}_i = 0$ so \eqref{eq: ellipsoidal approx: tight over-approx of BRS proof 5} holds with equality. Since $\norml{\hat{w}_i} \leq 1 $ implies that $\underline{e_i}(s, x^\star(s)) \leq 0$, from \eqref{eq: ellipsoidal approx: tight over-approx of BRS proof 4} and \eqref{eq: ellipsoidal approx: tight over-approx of BRS proof 5},
\begin{equation}
    -\tfrac{d}{ds}\underline{e_i}(s, x^\star(s)) \geq 0, \quad  \forall s \in \left\{s\in[t,T]\,|\,\underline{e_i}(s, x^\star(s)) \leq 0\right\},
    \label{eq: ellipsoidal approx: tight over-approx of BRS proof 6}
\end{equation}
which implies that $\underline{e_i}(s,x^\star(s)) \geq 0$ for all $s\in[t,T]$. To see this, suppose that the contrary is true, that is, $\underline{e_i}(s_1,x^\star(s_1)) < 0$ for some $s_1\in[t,T]$. Note that $\underline{e_i}(T, x^\star(T)) = g(x^\star(T))  = 0$, since $x^\star(T) \in \partial \mathcal{X}$ implies $\norml{Q^{-\half}_i(x^\star(T) -x_e)} = 1$. Thus, there exists a time $s_2\in(s_1,T]$ such that $\underline{e_i}(s_2, x^\star(s_2)) \geq 0$. Let $s_2$ be the first such time, i.e., $s_2 \doteq \min\{s \in (s_1,T]\,|\,\underline{e_i}(s,x^\star(s)) \geq 0\}$. Then, for all $s \in [s_1, s_2]$, $\underline{e_i}(s, x^\star(s)) \leq 0$. From \eqref{eq: ellipsoidal approx: tight over-approx of BRS proof 6},
\begin{align*}
    0 &= \underline{e_i}(s_2,x^\star(s_2)) = \underline{e_i}(s_1,x^\star(s_1)) + \integral{s_1}{s_2}{\tfrac{d}{ds}\underline{e_i}(s,x^\star(s))}{s}\leq \underline{e_i}(s_1,x^\star(s_1)),
\end{align*}
which yields the desired contradiction. Thus, $\underline{e_i}(s,x^\star(s)) \geq 0$ for all $s\in [t,T]$, which implies that $x^\star(t) \notin \{x\in\R^n\,|\,\underline{e_i}(t,x) < 0\} = \text{int}\left(\mathcal{E}(q_i(t), Q_i(t))\right)$. 

Finally, since $s\mapsto x^\star(s)$ is a solution of \eqref{eq: ellipsoidal approx: linear system} with $x^\star(T) \in \mathcal{X}$, by Definition \ref{def: background: BRS}, $x^\star(t)\in\mathcal{G}(t)$. From Theorem \ref{theorem: ellipsoidal approx: over-approx BRS}, $\mathcal{G}(t)\subseteq \mathcal{E}(q_i(t),Q_i(t))$, so $x^\star(t) \in \mathcal{E}(q_i(t),Q_i(t))$. Since $x^\star(t) \notin \text{int}\left(\mathcal{E}(q_i(t), Q_i(t))\right)$, then, $x^\star(t)\in\partial \mathcal{E}(q_i(t), Q_i(t))$. Furthermore, $x^\star(t)$ can not be an interior point of $\mathcal{G}(t)$ otherwise it would also be an interior point of $\mathcal{E}(q_i(t), Q_i(t))$. Thus, $x^\star(t) \in \partial \mathcal{G}(t) \cap \partial \mathcal{E}(q_i(t), Q_i(t))$ for all $t\in[0,T]$. 
\end{proof}

\begin{remark}
Theorems \ref{theorem: ellipsoidal approx: under-approx BRS} and \ref{theorem: ellipsoidal approx: over-approx BRS} can still be applicable for LTV systems where the input set $\U$ and the terminal set $\mathcal{X}$ are not ellipsoidal. In particular, if these sets are under-approximated by ellipsoids $\mathcal{E}(p,P)\subseteq \U$ and $\mathcal{E}(x_e, X_e)\subseteq \mathcal{X}$, which enter the FVP in \eqref{eq: ellipsoidal approx: under approx BRS dynamics + terminal condition}-\eqref{eq: ellipsoidal approx: under approx BRS Q star matrix}, then, assertion (a) of Theorem \ref{theorem: ellipsoidal approx: under-approx BRS} holds. Analogously, if $\U$ and $\mathcal{X}$ are over-approximated by ellipsoids, then, all assertions of Theorem \ref{theorem: ellipsoidal approx: over-approx BRS} hold.
    \label{remark: ellipsoidal approx: generalising input and terminal set}
\end{remark}
\begin{remark}
Lemma \ref{lemma: background: FRS from BRS} can be used to produce analogous results to this section for under- and over-approximating the forwards reachable set $\mathcal{F}(t)$ of \eqref{eq: ellipsoidal approx: linear system} subject to the initial condition $x(t) \in \mathcal{X} = \mathcal{E}(x_e, X_e)$ (see Definition \ref{def: background: FRS}). This would require that \eqref{eq: ellipsoidal approx: under approx BRS dynamics + terminal condition} and \eqref{eq: ellipsoidal approx: over approx BRS dynamics + terminal condition} be solved as initial value problems (IVP) with $q_i(t) = x_e$ and $Q_i(t) = X_e$ and the signs on $\underline{\calR_i}$ and $\overline{\calR}_i$ reversed to `$+\underline{\calR_i}$' and `$+\overline{\calR}_i$'. The resulting union of ellipsoids $\cup^{n_q}_{i=1}\mathcal{E}(q_i(T),Q_i(T))$ from the corresponding IVP of \eqref{eq: ellipsoidal approx: under approx BRS dynamics + terminal condition} and the intersection of ellipsoids $\cap^{n_q}_{i=1}\mathcal{E}(q_i(T),Q_i(T))$ from the corresponding IVP of \eqref{eq: ellipsoidal approx: over approx BRS dynamics + terminal condition} would be an under- and over-approximation of $\mathcal{F}(t)$, respectively. The equivalent tight ellipsoidal approximations of Corollaries \ref{corollary: ellipsoidal approx: tight under-approx of BRS} and \ref{corollary: ellipsoidal approx: tight over-approx of BRS} for $\mathcal{F}(t)$ would require that the solution $x^\star(\cdot)$ be obtained via conversion of \eqref{eq: ellipsoidal approx: PMP backwards dynamics for LTV system} and \eqref{eq: ellipsoidal approx: over approx boundary soln dynamics} to initial value problems.
\label{remark: ellipsoidal approx: under and over approx for FRS}
\end{remark}

\subsection{Numerical Implementation}
Corollaries \ref{corollary: ellipsoidal approx: tight under-approx of BRS} and \ref{corollary: ellipsoidal approx: tight over-approx of BRS} can be used to generate, respectively, a union of under-approximating ellipsoids and an intersection of over-approximating ellipsoids for \eqref{eq: ellipsoidal approx: linear system} whose boundaries coincide with $\mathcal{G}(t)$. A possible algorithmic procedure for doing so is summarised in Algorithms \ref{alg: tight under-approximating union} and \ref{alg: tight over-approximating intersection}, which, ignoring integration errors, produces a collection of under-approximating and over-approximating ellipsoids $\mathcal{E}(q_i, Q_i)$ and a corresponding set of boundary points $x^\star_i$ at time $t$.

\begin{algorithm}[t!]
 \caption{Tight Under-approximating Ellipsoids}
 \label{alg: tight under-approximating union}
\begin{algorithmic}[1]
  \STATE \textbf{Input:} $\{t, T\}$, $\Delta$, $\{A(\cdot), B(\cdot)\}$, $\mathcal{E}(p,P)$, $\mathcal{E}(x_e,X_e)$, $n_q$, and $Q_{min}$
  \STATE \textbf{Output:} $\{q_i, Q_i\}^{n_q}_{i=1}$ and $\{x^\star_i\}^{n_q}_{i=1}$
  \FORALL{$i \in \{1,\ldots,n_q\}$}
  \STATE $(q_i, Q_i) \leftarrow (x_e, X_e)$
  \STATE $x^\star_i \leftarrow Q^{\half}_i w_i + q_i$  \COMMENT{take any $w_i\in\R^n$ with unit norm}
  \STATE $\lambda_i \leftarrow 2X^{-1}_e(x^\star_i - x_e)$
  \ENDFOR
  \STATE $s \leftarrow T$
  \WHILE{$s > t$}
    \STATE $\Bar{A} \leftarrow A(s)$, $\Bar{B} \leftarrow B(s)$
    \FORALL{$i \in \{1,\ldots,n_q\}$}
    \STATE $(\hat{w}_i, \calQ_{i,\star})\leftarrow \left(Q^{-\half}_i(x^\star_i - q_i), \,Q^{-\half}_i\Bar{B}P\Bar{B}^TQ^{-\half}_i\right)$
    \STATE $u^\star_i \leftarrow \argmax_{u\in\mathcal{E}(p,P)}\iprod{-\Bar{B}^T\lambda_i}{u}$ \\ 
    \COMMENT{use \eqref{eq: ellipsoidal approx: maximising control for boundary solutions} with $l = \Bar{B}^T\lambda_i$}
    \IF{$\calQ^\half_{i,\star}\hat{w}_i = 0$ or $\underline{\Lambda}(Q_i) < Q_{min}$}
      \STATE $S_i \leftarrow \I$
    \ELSE
      \STATE $S_i \leftarrow$ solution of \eqref{eq: ellipsoidal approx: under approx BRS orthogonal matrix}  \COMMENT{use Remark \ref{remark: ellipsoidal approx: rodrigues rotation formula}}
    \ENDIF
    \STATE $(q_i, Q_i) \leftarrow  \texttt{getNextqQ}(\Delta, q_i, Q_i, \Bar{A}, \Bar{B}, S_i)$ 
    \STATE $(x^\star_i, \lambda_i)\leftarrow \texttt{getNextState}(\Delta, x^\star_i, 
    \lambda_i, \Bar{A}, \Bar{B}, u^\star_i)$\\
    \ENDFOR
        \STATE $s \leftarrow s - \Delta$
  \ENDWHILE
\end{algorithmic}
\end{algorithm}

\begin{algorithm}[t!]
 \caption{Tight Over-approximating Ellipsoids}
 \label{alg: tight over-approximating intersection}
\begin{algorithmic}[1]
  \STATE \textbf{Input:} $\{t, T\}$, $\Delta$, $\{A(\cdot), B(\cdot)\}$, $\mathcal{E}(p,P)$, $\mathcal{E}(x_e,X_e)$, $n_q$, and $\kappa_{min}$
  \STATE \textbf{Output:} $\{q_i, Q_i\}^{n_q}_{i=1}$ and $\{x^\star_i\}^{n_q}_{i=1}$
  \FORALL{$i \in \{1,\ldots,n_q\}$}
  \STATE $(q_i, Q_i) \leftarrow (x_e, X_e)$
  \STATE $x^\star_i \leftarrow Q^{\half}_i w_i + q_i$  {\COMMENT{take any $w_i\in\R^n$ with unit norm}}
  \ENDFOR
  \STATE $s \leftarrow T$
  \WHILE{$s > t$}
    \STATE $\Bar{A} \leftarrow A(s)$, $\Bar{B} \leftarrow B(s)$
    \FORALL{$i \in \{1,\ldots,n_q\}$}
    \STATE $(\hat{w}_i, \calQ_{i,\star})\leftarrow \left(Q^{-\half}_i(x^\star_i - q_i), \,Q^{-\half}_i\Bar{B}P\Bar{B}^TQ^{-\half}_i\right)$
    \STATE $u^\star_i \leftarrow \argmax_{u\in\mathcal{E}(p,P)}\iprod{-\Bar{B}^TQ^{-\half}_i\hat{w}_i}{u}$ \\ \COMMENT{use \eqref{eq: ellipsoidal approx: maximising control for boundary solutions} with $l = \Bar{B}^TQ^{-\half}_i\hat{w}_i$}
    \IF{$\calQ^\half_{i,\star}\hat{w}_i = 0$}
      \STATE $\kappa_i \leftarrow \kappa_{min}$
    \ELSE
      \STATE $\kappa_i \leftarrow \bnorml{\calQ^{\half}_{i,\star}\hat{w}_i}\slash \bnorml{\hat{w}_i}$ 
    \ENDIF
    \STATE $(q_i, Q_i) \leftarrow  \texttt{getNextqQ}(\Delta, q_i, Q_i, \Bar{A}, \Bar{B}, \kappa_i)$
    \STATE $x^\star_i \leftarrow \texttt{getNextState}(\Delta, x^\star_i, \Bar{A}, \Bar{B}, u^\star_i)$
    \ENDFOR
        \STATE $s \leftarrow s - \Delta$
  \ENDWHILE
\end{algorithmic}
\end{algorithm}

The `getNextqQ' and `getNextState' functions in steps 19-20 of Algorithm \ref{alg: tight under-approximating union} and steps 18-19 of Algorithm \ref{alg: tight over-approximating intersection} refer to any integration scheme, e.g., a Runge-Kutta scheme, that integrates, respectively, \eqref{eq: ellipsoidal approx: under approx BRS dynamics + terminal condition}, \eqref{eq: ellipsoidal approx: PMP backwards dynamics for LTV system}, \eqref{eq: ellipsoidal approx: over approx BRS dynamics + terminal condition}, and \eqref{eq: ellipsoidal approx: over approx boundary soln dynamics} backwards in time over a time span $\Delta \ll T-t$ with all time-varying elements held constant. To compute the maximising input in \eqref{eq: ellipsoidal approx: PMP backwards dynamics for LTV system} and \eqref{eq: ellipsoidal approx: over approx boundary soln dynamics}, one can use the fact that (e.g., \cite[Lemma 10]{LMD23})
\begin{equation}
    u^\star = \begin{cases}
        p - \frac{Pl}{\norml{P^{\half}l}},  & l \neq 0, \\ p, &l = 0,
    \end{cases} 
    \label{eq: ellipsoidal approx: maximising control for boundary solutions}
\end{equation}
is a maximiser of $\iprod{-l}{u}$ on $\mathcal{E}(p,P)$ for any $l\in\R^m$.

In Algorithm \ref{alg: tight under-approximating union}, we select $S_i = \I$ when the squared length of the shortest axis of $\mathcal{E}(q_i,Q_i)$ is below some tolerance $Q_{min}>0$, which prevents the ellipsoid $\mathcal{E}(q_i, Q_i)$ from becoming degenerate (see Remark \ref{remark: ellipsoidal approx: positive definiteness of ellipsoid shape}). Similarly, to ensure that $\kappa_i > 0$ in  Algorithm \ref{alg: tight over-approximating intersection}, we select $\kappa_i = \kappa_{min} >0$ in the case where $ \calQ^{\half}_{i,\star}\hat{w}_i= 0$. Note that this subsumes the case where $\hat{w}_i = 0$. If these tolerances are not met at any time $s\in(t,T)$, then the boundaries of $\mathcal{E}(q_i, Q_i)$ and the backwards reachable set will coincide (to the tolerance level of the chosen integration scheme) at the points $x^\star_i$.

%% file: text/numerical_examples.tex
\section{Numerical Example}
\label{sec: numerical examples}
Let us now illustrate how the results of Section \ref{sec: LTV reachability} can be used to generate under- and over-approximations of the backwards reachable set for \eqref{eq: ellipsoidal approx: linear system} using a numerical example. We consider a forced parametric oscillator whose dynamics can be described by the LTV state-space model
\begin{align}
    &\Dot{x}(s) = \begin{bmatrix}0&1\\-\omega^2(s)&0\end{bmatrix}x(s) + \begin{bmatrix}0 \\ 1\end{bmatrix}u(s),\quad \text{a.e. }s\in(0,1.5), \nn\\
    &x(1.5) \in \mathcal{X}, 
\label{eq: numerical ex: forced oscillator system description}
\end{align}
where $\omega(s) = \sqrt{4 + 2 \cos{(2s)}}$ is the frequency, $x(s)=[x_1(s), x_2(s)]^T\in\mathbb{R}^2$ is the state, and $u(s)\in\mathbb{U}\subset\mathbb{R}$ is the input. We take the input set $\mathbb{U}$ and the terminal set $\mathcal{X}$ to be
\begin{equation}
    \mathbb{U} \doteq \mathcal{E}\left(0,1\right), \quad  \mathcal{X} \doteq \mathcal{E}\left(0,0.1^2 \I\right).
    \label{eq: numerical ex: forced oscillator ellipsoidal sets}
\end{equation}

For comparison purposes, the `exact' backwards reachable set $\mathcal{G}(t)$ of \eqref{eq: numerical ex: forced oscillator system description} was numerically computed using the grid-based, finite-difference scheme provided by \cite{mitchelltoolbox}. A grid of $250\times 250$ points was used to compute $\mathcal{G}(t)$ and is included in the dashed black lines in Fig. \ref{fig: numerical ex: under approx BRS} and Fig. \ref{fig: numerical ex: over approx BRS}.

\begin{figure}[ht!]
    \centering
    \includegraphics[width=0.7\columnwidth]{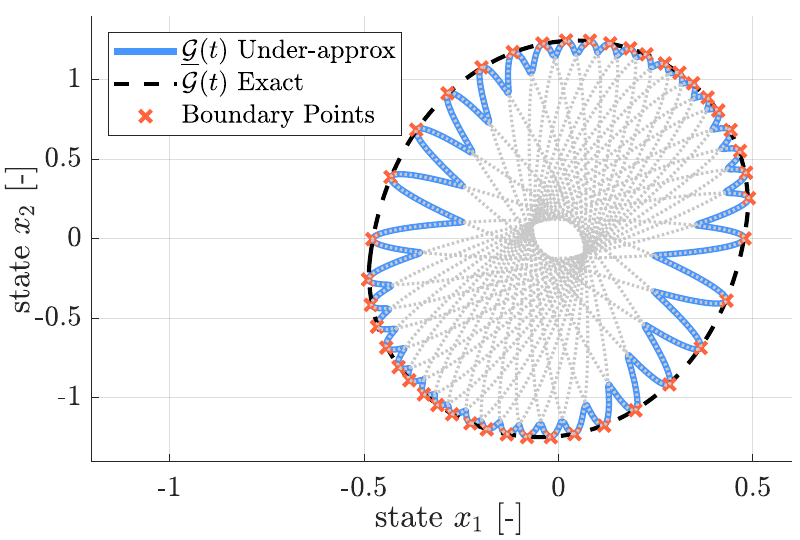}
    \caption{Comparison of the union of ellipsoids $\underline{\mathcal{G}}(t)$ and the backwards reachable set $\mathcal{G}(t)$ for \eqref{eq: numerical ex: forced oscillator system description} at time $t = 0$. Grey dotted lines give the outline of each individual ellipsoid in $\underline{\mathcal{G}}(t)$. Animation: \texttt{https://youtu.be/XNGjqlYKxnU}.}
    \label{fig: numerical ex: under approx BRS}
\end{figure}
 \begin{figure}[ht!]
    \centering
    \includegraphics[width=0.7\columnwidth]{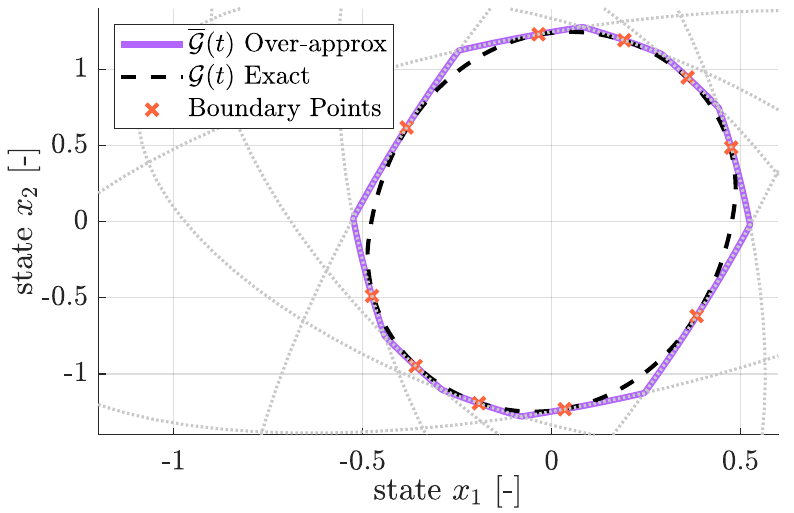}
    \caption{Comparison of the intersection of ellipsoids $\overline{\mathcal{G}}(t)$ and the backwards reachable set $\mathcal{G}(t)$ for \eqref{eq: numerical ex: forced oscillator system description} at time $t = 0$. Grey dotted lines give the outline of each individual ellipsoid in $\overline{\mathcal{G}}(t)$. Animation: \texttt{https://youtu.be/-0sMzyfY2GY}.}
    \label{fig: numerical ex: over approx BRS}
\end{figure}

A union of $n_q = 21$ under-approximating ellipsoids $\underline{\mathcal{G}}(t) \doteq \cup^{n_q}_{i=1}\mathcal{E}(q_i, Q_i)$ for \eqref{eq: numerical ex: forced oscillator system description} was produced using Algorithm \ref{alg: tight under-approximating union}. A time step of $\Delta = 0.01$s was selected with a minimum squared ellipsoid length of $Q_{min} = 10^{-4}$. The initial set of points $\{x^\star_i\}^{n_q}_{i=1}$ was uniformly distributed on the boundary of $\mathcal{X}$. The boundary for $\underline{\mathcal{G}}(t)$ at time $t = 0$ is presented in Fig. \ref{fig: numerical ex: under approx BRS} alongside the set of points $\{x^\star_i\}^{n_q}_{i=1}$ which lie on the boundary of each ellipsoid $\mathcal{E}(q_i, Q_i)$. Due to symmetry of ellipsoids about their centre, if $x^\star_i \in \partial \mathcal{E}(q_i,Q_i)$ then $2q_i - x^\star_i \in \partial\mathcal{E}(q_i,Q_i)$, so these points have also been included.

Using Algorithm \ref{alg: tight over-approximating intersection}, an intersection of $n_q = 5$ over-approximating ellipsoids $\overline{\mathcal{G}}(t) \doteq \cap^{n_q}_{i=1}\mathcal{E}(q_i, Q_i)$ for \eqref{eq: numerical ex: forced oscillator system description} was also produced with the same time step. A lower bound of $\kappa_{min} = 10^{-4}$ was selected with the initial set of points $\{x^\star_i\}^{n_q}_{i=1}$ also uniformly distributed on $\partial\mathcal{X}$. The boundary of $\overline{\mathcal{G}}(t)$ at time $t=0$ is displayed in Fig. \ref{fig: numerical ex: over approx BRS}. 

As guaranteed by Corollaries \ref{corollary: ellipsoidal approx: tight under-approx of BRS} and \ref{corollary: ellipsoidal approx: tight over-approx of BRS}, the collection of boundary points $\{x^\star_i\}^{n_q}_{i=1}$ for both the under- and over-approximating collection of ellipsoids coincide with the boundary of $\mathcal{G}(t)$, yielding tight under- and over-approximations of the backwards reachable set. Note that in this example and in other tests, we have found that the condition $\underline{\Lambda}(Q_i) < Q_{min}$ in step 14 of Algorithm \ref{alg: tight under-approximating union} is never met and the condition $\calQ^\half_{i,\star}\hat{w}_i= 0$ in step 13 of Algorithm \ref{alg: tight over-approximating intersection} is not met so long as $x^\star_i$ is initialised with $\calQ^{\half}_{i,\star}\hat{w}_i \neq 0$. Further analysis is required to state conditions for which we can expect such a scenario to occur.

Using a fourteen-core Intel{\textregistered} Core\textsuperscript{\texttrademark} i7-13700H CPU, computation times in \texttt{MATLAB} for generating the sets $\mathcal{G}(t)$, $\underline{\mathcal{G}}(t)$, and $\overline{\mathcal{G}}(t)$ are displayed in Table \ref{tab: numerical ex: computation times}. Their absolute and relative internal areas, which have been computed via the `trapz' routine in \texttt{MATLAB}, are also displayed. We observe that Algorithms \ref{alg: tight under-approximating union} and \ref{alg: tight over-approximating intersection} offer substantial computational savings over grid-based approaches for approximating the reachable set of \eqref{eq: numerical ex: forced oscillator system description} whilst capturing a comparable area. 

To investigate the scalability of Algorithms \ref{alg: tight under-approximating union} and \ref{alg: tight over-approximating intersection} with respect to the number of states, consider an $n$-state, parameterised LTV system described by
\begin{align}
        &\Dot{x}_1(s) = -(0.2\cos{(\gamma_1 s)} + 0.5)x_1(s) + \eta_1 x_2(s),\nn\\
        &\hspace{3.8cm} \vdots \nn\\
        &\Dot{x}_{n-1}(s) = -(0.2\cos{(\gamma_{n-1} s)} + 0.5)x_{n-1}(s) + \eta_{n-1} x_n(s), \nn\\
        &\Dot{x}_n(s) = u(s),
    \label{eq: numerical ex: timing analysis system description}
\end{align}
a.e. $s\in(0,2)$, where the parameters $\gamma_i, \eta_i \in \R$ are randomly selected from a uniform distribution over $[-0.5,\,0.5]$ for all $i\in\{1,\ldots,n-1\}$. Additionally, the input and terminal set for \eqref{eq: numerical ex: timing analysis system description} is selected to be
\begin{equation}
    \mathbb{U} \doteq \mathcal{E}\left(0,1\right)\subset\R , \quad  \mathcal{X} \doteq \mathcal{E}\left(0,\I\right)\subset \R^n.
    \label{eq: numerical ex: timing analysis ellipsoidal sets}
\end{equation}

\renewcommand{\arraystretch}{1.5}
\begin{table}[!ht]
\begin{center}
\caption{Comparison of computation times and internal areas (at $t=0$) for $\mathcal{G}(t)$, $\underline{\mathcal{G}}(t)$, and $\overline{\mathcal{G}}(t)$.}
\label{tab: numerical ex: computation times}
\begin{tabular}{|c|c|c|} \hline
\cellcolor{lightgray} \textbf{Numerical method} & \cellcolor{lightgray} \textbf{Computation time} [s]  & \cellcolor{lightgray} \textbf{Area (rel. to $\mathcal{G}(t)$)} \\ \hline \hline
Grid-based \cite{mitchelltoolbox} $\mathcal{G}(t)$& 19.17 & 1.880 (100.0\%)\\ \hline
Algorithm \ref{alg: tight under-approximating union} $\underline{\mathcal{G}}(t)$ & $0.1875$& 1.694 (90.09\%)\\ \hline
Algorithm \ref{alg: tight over-approximating intersection} $\overline{\mathcal{G}}(t)$ & $0.09375$& 1.974 (105.0\%)\\ \hline
\end{tabular}
\end{center}
\end{table}
\renewcommand{\arraystretch}{1}
The average computation time over 30 trials for generating $\underline{\mathcal{G}}(t)$ and $\overline{\mathcal{G}}(t)$ for \eqref{eq: numerical ex: timing analysis system description}-\eqref{eq: numerical ex: timing analysis ellipsoidal sets} is displayed on a log scale in Fig. \ref{fig: numerical ex: timing comparison} for $n = 2$ to $50$ states. A fixed number of ellipsoids $n_q = 20$ and a time step of $\Delta = 0.01$s was used for both Algorithms \ref{alg: tight under-approximating union} and \ref{alg: tight over-approximating intersection}. Also displayed are the computation times for the exact reachable set $\mathcal{G}(t)$ using 80 grid points per dimension from $n = 2$ to $4$ states. The computation time of Algorithms \ref{alg: tight under-approximating union} and \ref{alg: tight over-approximating intersection} is only explicitly tied to the number of states through the integration of a matrix differential equation, which has a computational complexity of $\mathcal{O}(n^3)$. This sub-exponential scaling is consistent with Fig. \ref{fig: numerical ex: timing comparison} where the $ n = 50$ case takes under 5 seconds to compute for both $\underline{\mathcal{G}}(t)$ and $\overline{\mathcal{G}}(t)$.  
\begin{figure}[h!]
    \centering
\includegraphics[width=0.65\columnwidth]{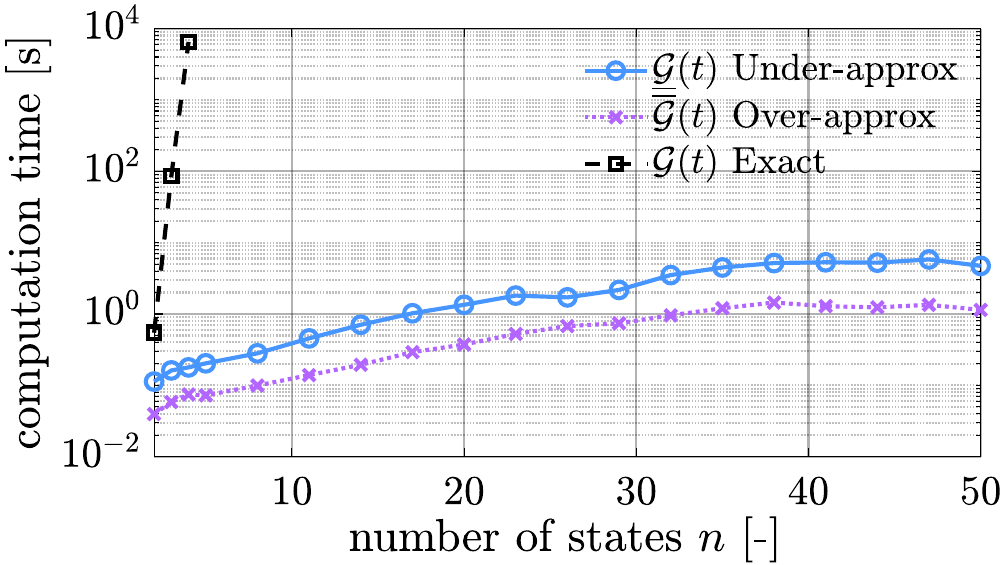}
    \caption{Comparison of computation times for the sets $\underline{\mathcal{G}}(t)$, $\overline{\mathcal{G}}(t)$, and $\mathcal{G}(t)$ as the number of states $n$ increases.}
    \label{fig: numerical ex: timing comparison}
\end{figure}

The flexibility in deciding how many ellipsoids are used in generating the under- and over-approximating sets $\underline{\mathcal{G}}(t)$ and $\overline{\mathcal{G}}(t)$ enables a trade-off to be made between the computational expense and the conservatism of the sets. Although more ellipsoids are likely required to obtain a similar degree of conservatism to Fig. \ref{fig: numerical ex: under approx BRS} and Fig. \ref{fig: numerical ex: over approx BRS} as $n$ increases, the guarantees provided in the results of Sections \ref{sec: approximate reachability} and \ref{sec: LTV reachability} will hold regardless of the number of ellipsoids chosen. In particular, states in $\underline{\mathcal{G}}(t)$ at any time $t\in[0,T]$ retain the guarantee that a constraint admissible control leading to the terminal set $\mathcal{X}$ exists. Analogously, states outside of $\overline{\mathcal{G}}(t)$ will always possess a certificate of safety, guaranteeing that the system will avoid the potentially dangerous terminal set. Approximation schemes, such as those offered here, enable reachability-based control strategies to be employed in applications where such strategies would traditionally be deemed intractable.

%% file: text/conclusion.tex
\section{Conclusion}
\label{sec: conclusion}
In this paper, we have demonstrated that viscosity supersolutions and subsolutions of a Hamilton-Jacobi-Bellman equation defined over a local domain can be used to generate under- and over-approximating reachable sets for time-varying nonlinear systems. This observation was then used to develop numerical schemes for approximating either the backwards or forwards reachable set of linear time-varying (LTV) systems. 

A union and intersection of ellipsoids, which can be characterised by the minimum and maximum over a collection of quadratic functions, was proposed to under- and over-approximate the reachable set, respectively. The under-approximating ellipsoids admit a generalisation that allows for any solution of the LTV system to be contained within it. This in turn implies that it is possible to have the boundaries of the under-approximating ellipsoids and the exact reachable set coincide along a solution of the system. The over-approximating ellipsoids also admit an analogous generalisation, allowing for a tight over-approximation of the reachable set. The numerical example presented in this paper highlights a marked computational advantage afforded to LTV systems compared to grid-based approaches. Moreover, by utilising a set of conditions that are applicable to nonlinear systems and a general class of sets, the work presented here acts as a demonstration for how approximation schemes may be developed for other system classes and set representations. 

%% file: text/appendices.tex
\section*{Appendix}
\noindent \textbf{Proof of Lemma \ref{lemma: approx reach: closure and interior of sublevel sets}: }
     First consider the statement in \eqref{eq: approx reach: closure of strict sublevel set}. Let $\mathscr{F}^z_{0} \doteq z^{-1}(\{0\})$, $\mathscr{F}^z_{<0}\doteq z^{-1}((-\infty,0))$, and $\mathscr{F}^z_{\leq0} \doteq z^{-1}((-\infty, 0])$. Fix any $\Bar{x}\in\mathscr{F}^z_{0}$. As $z\in\cZ$, there exists a sequence of points $\{x_k\}^\infty_{k=1}$ such that for each $k\in\N$, $x_k \in \B_{r_k}(\Bar{x})$ with $r_k \doteq \textstyle{\frac{1}{k}}$ and $z(x_k) < 0$. Since $\lim_{k\rightarrow\infty}\norm{x_k - \Bar{x}} = 0$ and $x_k \in \mathscr{F}^z_{<0}$ for all $k\in\N$, $\Bar{x}$ is a limit point of $\mathscr{F}^z_{<0}$ so $\Bar{x}\in \text{cl}\left(\mathscr{F}^z_{<0}\right)$. As $\Bar{x}$ was arbitrarily selected in $\mathscr{F}^z_{0}$, $\mathscr{F}^z_{0}\subseteq \text{cl}\left(\mathscr{F}^z_{<0}\right)$, which implies $\mathscr{F}^z_{\leq0} \subseteq \text{cl}\left(\mathscr{F}^z_{<0}\right)$. Noting that the closure of sets preserves subset relationships and $\mathscr{F}^z_{\leq0}$ is closed, $\text{cl}\left(\mathscr{F}^z_{<0}\right)\subseteq \text{cl}(\mathscr{F}^z_{\leq0})=\mathscr{F}^z_{\leq0}$, thus $ \text{cl}\left(\mathscr{F}^z_{<0}\right) = \mathscr{F}^z_{\leq0}$.

Next, consider the statement in \eqref{eq: approx reach: interior of non-strict sublevel set}. Since $z$ is of class $\cZ$, 
\begin{equation}
 \mathscr{F}^z_0 \cap \text{int}\left(\mathscr{F}^z_{\leq0}\right) =\emptyset.
    \label{eq: appendix: closure and interior of sublevel sets proof 2}
\end{equation}
This follows because for any $\Bar{x}\in\mathscr{F}^z_0$ there can not exist an $\Bar{r}>0$ such that $\B_{\Bar{r}}(\Bar{x})\subseteq \mathscr{F}^z_0$. Using the fact that the interior of sets preserves subset relationships, $\text{int}\left(\mathscr{F}^z_{<0}\right) \subseteq \text{int}\left(\mathscr{F}^z_{\leq0}\right)$. Since the interior of a set is the largest open subset and $\mathscr{F}^z_{<0}$ is an open set, it follows that $\mathscr{F}^z_{<0} \subseteq \text{int}\left(\mathscr{F}^z_{\leq0}\right)$. Moreover, for any $\Bar{x} \in\text{int}\left(\mathscr{F}^z_{\leq0}\right)$ either $\Bar{x}\in \mathscr{F}^z_{0}$ or $\mathscr{F}^z_{<0}$, which means \eqref{eq: appendix: closure and interior of sublevel sets proof 2} implies $\text{int}\left(\mathscr{F}^z_{\leq0}\right)\subseteq \mathscr{F}^z_{<0}$. Thus, $\mathscr{F}^z_{<0} = \text{int}\left(\mathscr{F}^z_{\leq0}\right)$.  \hfill{\qedsymbol}
    
\vspace{0.4cm}
\noindent\textbf{Proof of Lemma \ref{lemma: approx reach: value function satisfies attached sublevel set}: }
Fix any $t \in [0,T]$ and any point $x_0\in\{x\in\R^n\,|\,v(t,x) = 0\}$. Since $f$ satisfies Assumption \ref{assumption: background: convex flow field}, the infimum in \eqref{eq: background: value function} is attained (use \cite[Theorem 7.1.4, p.189]{CS:04}). In particular, there exists $u^\star\in\cU[t,T]$ such that $g(\varphi(T;t, x_0, u^\star)) = v(t, x_0) = 0$. Let $x^\star(\cdot) \doteq \varphi(\cdot;t,x_0,u^\star)$. Since $x^\star(T) = \varphi(T;t,x_0, u^\star)$ is arbitrarily close to a point in $g^{-1}((-\infty,0))$, $x_0$ must also be arbitrarily close to a point in $\{x\in\R^n\,|\,v(t,x) < 0\}$. To see this, first note that by uniqueness of solutions of \eqref{eq: background: nonlinear system},
\begin{equation*}
   x^\star(s) \doteq  \varphi(s; t, x_0,u^\star) = \varphi(s; T, x^\star(T), u^\star), \quad \forall s\in[t,T],
\end{equation*}
since $x^\star(\cdot)$ and $\varphi(\cdot;T,x^\star(T),u^\star)$ share the same control and terminal condition $x^\star(T)$. Then, by continuity of solutions of \eqref{eq: background: nonlinear system}, for any $ r > 0$ there exists a $\Bar{r} > 0$ such that
\begin{equation}
    \norm{\varphi(t; T, \Bar{x}, u^\star) -x^\star(t)} =   \norm{\varphi(t; T, \Bar{x}, u^\star) - x_0} < r,
    \label{eq: approx reach: value function attached sublevel set proof 2}
\end{equation}
whenever $\norm{\Bar{x} - x^\star(T)} < \Bar{r}$. Moreover, since $x^\star(T) \in g^{-1}(\{0\})$ and $g$ is of class $\cZ$, given any $\Bar{r} > 0$, the set $\B_{\Bar{r}}(x^\star(T)) \cap g^{-1}((-\infty,0))$ is non-empty. Then, take any $\Bar{x} \in \B_{\Bar{r}}(x^\star(T)) \cap g^{-1}((-\infty,0))$ to obtain
\begin{align}
    &v(t,\varphi(t; T, \Bar{x}, u^\star)) =\inf_{u\in\cU[t, T]}g\left(\varphi(T;t, \varphi(t; T, \Bar{x}, u^\star), u)\right)\nn\\
    &\qquad \leq g\left(\varphi(T;t, \varphi(t; T, \Bar{x}, u^\star), u^\star)\right) = g(\Bar{x}) < 0.        \label{eq: approx reach: value function attached sublevel set proof 3}
\end{align}
From \eqref{eq: approx reach: value function attached sublevel set proof 2} and \eqref{eq: approx reach: value function attached sublevel set proof 3}, it follows that for any $r>0$ there exists a point in $\B_r(x_0)$ that also lies in $\{x\in\R^n\,|\,v(t, x) < 0\}$. As $x_0\in\{x\in\R^n\,|\, v(t,x) = 0\}$ is arbitrary, the map $x\mapsto v(t,x)$ is of class $\cZ$ for all $t\in[0,T]$. \hfill{\qedsymbol}

\vspace{0.4cm}
\noindent\textbf{Proof of Lemma \ref{lemma: ellipsoidal approx: min over quadratics attached sublevel set}: }
Fix any $t\in[0,T]$ and $\bar{x} \in \{x\in\R^n\,|\, \overline{e}(t,x) = 0\}$. As $Q_i(t)\in \spd{n}$ implies $Q_i(t)$ is invertible, $\overline{e}$ in \eqref{eq: ellipsoidal approx: proposed under-approx value function} is well-defined. Omitting time-dependence for brevity, there exists an $\bar{n}\doteq \bar{n}(t,\bar{x})\in\{1,\ldots,n_q\}$ such that
\begin{equation}
    \iprod{\bar{x}-q_{\bar{n}}}{Q^{-1}_{\bar{n}}(\bar{x}- q_{\bar{n}})} - 1 = 0 \iff \bnorml{Q^{-\frac{1}{2}}_{\bar{n}}(\bar{x} - q_{\bar{n}})} = 1.
    \label{eq: ellipsoidal approx: min max over quadratics proof 1}
\end{equation}
Fix any $r > 0$ and let $\Tilde{x} \doteq \bar{x} - \frac{a}{2}\cdot \frac{\bar{x} - q_{\bar{n}}}{\norml{\bar{x} - q_{\bar{n}}}}$ with $a \doteq \min\{r, \norml{\bar{x} - q_{\bar{n}}}\} > 0$. Note that $\Tilde{x}$ is well defined as $\bar{x} \neq q_{\bar{n}}$. Since $0<a \leq \bnorml{\bar{x} - q_{\bar{n}}}$, $\tfrac{1}{2}\leq 1 - \frac{a}{2 \norml{\bar{x} - q_{\bar{n}}}} <1$. Then, from \eqref{eq: ellipsoidal approx: min max over quadratics proof 1} and the definition of $\Tilde{x}$,
\begin{align*}
    \bnorml{Q^{-\frac{1}{2}}_{\bar{n}}(\Tilde{x} - q_{\bar{n}})} &= \Big(1 - \tfrac{a}{2 \norml{\bar{x} - q_{\bar{n}}}}\Big)\bnorml{Q^{-\frac{1}{2}}_{\bar{n}}(\bar{x} - q_{\bar{n}})}<1\\
    \implies &\iprod{\Tilde{x}-q_{\bar{n}}}{Q^{-1}_{\bar{n}}(\Tilde{x}- q_{\bar{n}})} < 1.
\end{align*}
As defined in \eqref{eq: ellipsoidal approx: proposed under-approx value function}, $\overline{e}(t,\Tilde{x}) \leq \iprod{\Tilde{x}-q_{\bar{n}}}{Q^{-1}_{\bar{n}}(\Tilde{x}- q_{\bar{n}})} -1< 0$. Since $\norml{\Tilde{x} - \bar{x}}= \frac{a}{2} < r$, which means $\Tilde{x}\in\B_r(\bar{x})$, it follows that for any $\bar{x}\in \{x\in\R^n\,|\, \overline{e}(t,x) = 0\}$ and any $r>0$, there exists a point $\Tilde{x} \in \B_r(\bar{x})$ with $\overline{e}(t,\Tilde{x}) < 0$. Hence, as $t$ is arbitrary, $x\mapsto \overline{e}(t,x)$ is of class $\cZ$ for all $t\in[0,T]$. \hfill{}{\qedsymbol}